\documentclass[11pt]{amsart}

\usepackage{amsmath}
\usepackage{amsthm}
\usepackage{amssymb}
\usepackage{amscd}          			
\usepackage{bbm}            			
\usepackage{mathrsfs}           		
\usepackage{mathalfa}

\usepackage{graphicx}          		

\usepackage{verbatim}           		
\usepackage{enumitem}      		

\theoremstyle{plain}
\newtheorem{Theorem}{Theorem}

\newtheorem{theorem}[Theorem]{Theorem}
\newtheorem{proposition}[Theorem]{Proposition}

\newtheorem{corollary}[Theorem]{Corollary}

\newtheorem{lemma}[Theorem]{Lemma}

\theoremstyle{definition}

\newtheorem{example}[Theorem]{Example}
\newtheorem{definition}[Theorem]{Definition}
\newtheorem{remark}[Theorem]{Remark}

\theoremstyle{remark}

\usepackage{times}

\begin{document}
\title{Weak convergence of topological measures}
\author{S. V. Butler, University of California, Santa Barbara } 
\date{May 20, 2020}
\begin{abstract}
Topological measures and deficient topological measures are defined on open and closed subsets of a topological space, 
generalize regular Borel measures, and correspond to (non-linear in general) functionals that are linear on singly generated subalgebras 
or singly generated cones of functions.
They lack subadditivity, and many standard techniques 
of measure theory and functional analysis do not apply to them. 
Nevertheless, we show that many classical results of probability theory hold for topological and deficient topological measures. In particular, 
we prove a version of Aleksandrov's Theorem for equivalent definitions of weak convergence of deficient topological measures. 
We also prove a version of Prokhorov's Theorem which relates the existence of a weakly convergent subsequence 
in any sequence in a family of topological measures  
to the characteristics of being a uniformly bounded in variation and uniformly tight family.  
We define Prokhorov and Kantorovich-Rubenstein metrics and show that convergence in either of them   
implies weak convergence of (deficient) topological measures on metric spaces.  
We also generalize many known results about various dense and nowhere dense subsets of deficient topological measures.  
The present paper constitutes a first step to further research in probability theory and its applications in the context of topological measures and 
corresponding non-linear functionals.
\end{abstract}

\maketitle

\section{Introduction}

The origins of the theory of quasi-linear functionals and topological measures lie in mathematical axiomatization and interpretations of quantum physics
(\cite{vonN}, \cite{MackeyPaper}, \cite{MackeyBook}, \cite{Kadison}). In J. von Neumann's axiomatization of quantum mechanics, physical observables can be represented by the space $\mathcal{L}$ 
of Hermitian operators on a complex Hilbert space. The state of a physical system is represented by a positive normalized linear functional on
 $\mathcal{L}$. Some physicists, however, argued that the linearity of the functional, $\rho(A + B) = \rho(A) + \rho(B), \, A, B \in \mathcal{L} $,
 makes sense if observables $A$ and $B$ are simultaneously measurable, which means that $A,B$ are polynomials of the same $C \in \mathcal{L} $, 
 so $A,B$ belong to the subalgebra of $\mathcal{L} $ generated by $C$.
Mathematical interpretations of quantum physics by G. W. Mackey and R. V. Kadison 
led to very interesting mathematical problems, including 
the extension problem for probability measures in von Neumann algebras. 
This extension problem may be regarded as a special case of the linearity problem for physical states, 
which is closely related to the existence of quasi-linear functionals. 
J. F. Aarnes \cite{Aarnes:TheFirstPaper} introduced quasi-linear functionals (that are not linear) on $C(X) $ for a compact Hausdorff space $X$
and corresponding set functions, generalizing measures (initially called quasi-measures, now topological measures). 
He connected the two by establishing a representation theorem.  Aarnes's quasi-linear functionals 
are functionals that are linear on singly generated subalgebras, but (in general) not linear. 
For more information about physical interpretation of 
quasi-linear functionals see \cite{EntovPolterovich}, \cite{EnPoltZap}, \cite{Entov}, \cite{PoltRosenBook}, \cite{Aarnes:PhysStates69}, \cite{Aarnes:QuasiStates70}, \cite{Aarnes:TheFirstPaper}.
 
M. Entov and L. Polterovich first linked the theory of quasi-linear functionals to symplectic topology. 
They introduced symplectic quasi-states and partial symplectic quasi-states (\cite{EntovPolterovich}), which are subclasses of quasi-linear functionals.
(On a symplectic manifold that is a closed oriented surface every normalized quasi-linear functional is a symplectic quasi-state, 
see \cite[Chapter 5]{PoltRosenBook}). Article \cite{EntovPolterovich} was followed by numerous papers and a monograph \cite{PoltRosenBook},
and many authors have investigated and used various aspects of symplectic quasi-states and topological measures: 
their properties, their connection to spectral numbers and homogeneous quasi-morphisms, 
ways of constructing and approximating symplectic quasi-states, etc.  
Symplectic quasi-states can be used as a measurement of Poisson commutativity, 
and topological measures can be used to 
distinguish Lagrangian knots that have identical classical invariants (\cite[Chapters 4,6]{EntovPolterovich}). 
Symplectic quasi-states and topological measures play an important role in function theory on symplectic manifolds. 

Deficient topological measures are generalizations of topological measures. They were first defined and used by 
A. Rustad and O. Johansen  (\cite{OrjanAlf:CostrPropQlf}) and later independently reintroduced and further developed by  M. Svistula  
(\cite{Svistula:Signed},  \cite{Svistula:DTM}).  Deficient topological measures are not only interesting by themselves, but also provide an essential
framework for studying topological measures and quasi-linear functionals. 
Topological measures and deficient topological measures generalize regular Borel measures
and correspond to functionals that are linear on singly generated subalgebras or singly generated cones of functions. 
These non-linear functionals can be 
described in several ways, including symmetric and asymmetric Choquet integrals, see \cite[pp. 62, 87]{DD}   
and \cite[Corollary 8.5, Theorem 8.7, Remark 8.11]{Butler:ReprDTM}.
Deficient topological measures are not supermodular, and their domains 
are not closed under intersection and union; for these and other reasons, results of Choquet theory do not automatically translate 
for functionals representing deficient topological measures. 
It is interesting that, with different proof methods, one may obtain results that are typical for, stronger than, or strikingly different 
from Choquet theory results.  

Topological measures and deficient topological measures are defined on open and closed subsets of a topological space, 
which means that there is no algebraic structure on the  domain.  
They lack subadditivity and other properties typical for measures, and many standard techniques 
of measure theory and functional analysis do not apply to them. 
Nevertheless, we show that many classical results of probability theory hold for topological and deficient topological measures. In particular, 
we prove versions of 
Aleksandrov's Theorem for equivalent definitions of weak convergence of  topological and deficient topological measures. 
We also prove a version of Prokhorov's Theorem which relates the existence of a weakly convergent subsequence in any sequence
in a family of topological measures  
to the characteristics of being a uniformly bounded in variation and uniformly tight family.  
We define Prokhorov and Kantorovich-Rubenstein metrics and show that convergence in either of them   
implies weak convergence of deficient topological measures.  
We also generalize many known results about various dense and nowhere dense subsets of deficient topological measures.  

The present paper constitutes a first step to further research in probability theory and its applications in the context of topological measures and 
corresponding non-linear functionals.

In this paper $X$ is a locally compact space Hausdorff space.
By $C(X)$ we denote the set of all real-valued continuous functions on $X$ with the uniform norm, 
by $C_0(X)$ the set of continuous functions on $X$ vanishing at infinity, 
by $C_c(X)$ the set of continuous functions with compact support, and  by 
$C_0^+(X)$ the collection of all nonnegative functions from $C_0(X)$.

When we consider maps into extended real numbers we assume that any such map is not identically $\infty$. 

We denote by $\overline E$ the closure of a set $E$, and by $ \bigsqcup$ a union of disjoint sets.
A set $A \subseteq X$ is called bounded if $\overline A$ is compact. 
We denote by $id$ the identity function $id(x) = x$, 
and by $1_K$ the characteristic function of a set $K$. By $ supp \,  f $ we mean $ \overline{ \{x: f(x) \neq 0 \} }$.
We say that $Y$ is dense in $Z$ if $Z \subseteq \overline Y$.

Several collections of sets are used often.   They include:
$\mathscr{O}(X)$;
$\mathscr{C}(X)$; and
$\mathscr{K}(X)$-- 
the collection of open subsets of $X$;  the collection of closed subsets of   $X $;
and the collection of compact subsets of   $X $, respectively.

\begin{definition} \label{MDe2}
Let $X$ be a  topological space and $\nu$ be a set function on a family $\mathcal{E}$ of subsets of $X$ that 
contains $\mathscr{O}(X) \cup \mathscr{C}(X)$
with values in $[0, \infty]$. 
We say that 
\begin{itemize}
\item
$\nu$ is compact-finite if $ \nu(K) < \infty$ for any $ K \in \mathscr{K}(X)$;
\item
$\nu$ is simple if it only assumes  values $0$ and $1$;
\item
$ \nu$ is finite if $ \nu(X) < \infty$; 
\item
$\nu$ is inner regular (or inner compact regular) 
if $\nu(A) = \sup \{ \nu(C) :  C \subseteq A, C \in \mathscr{K}(X)\}$ for $A \in \mathcal{E}$;
\item
$\nu$ is inner closed regular  
if $\nu(A) = \sup \{ \nu(C) : C  \subseteq A, C \in  \mathscr{C}(X) \}$ for $A \in \mathcal{E}$;
\item
$\nu$ is outer regular  if 
$\nu(A) = \inf \{ \nu(U) : A  \subseteq U, U \in  \mathscr{O}(X) \}$   for $A \in \mathcal{E}$.

\end{itemize}
\end{definition}

\begin{definition}
A measure on $X$ is a countably additive set function on a $\sigma$-algebra of subsets of $X$ with values in $[0, \infty]$.  
A Borel measure on $X$ is a measure on the Borel $\sigma$-algebra on $X$.  
A Radon measure  $m$  on $X$ is  a compact-finite Borel measure that is outer regular on all Borel sets, and inner regular on all open sets, i.e.
$m(K) < \infty$ for every compact $K$, 
$ m(E) = \inf \{ m(U): E \subseteq U, U \text{  is open} \} $ for every Borel set $E$, and 
$m(U) = \sup \{  m(K): K \subseteq U, K  \text{  is compact} \}$ for every open set $U$. 
For a Borel measure  $m$ that is inner regular on all open sets (in particular, for a Radon measure)  we define $supp \ m$,  the support of  $m$, to be 
the complement of the largest open set $W$ such that $m(W) =0$. 
\end{definition}

For the following fact see, for example,  ~\cite[Chapter XI, 6.2]{Dugundji} and \cite[Lemma 7]{Butler:TMLCconstr}.
\begin{lemma} \label{easyLeLC}
Let $K \subseteq U, \ K \in \mathscr{K}(X),  \ U \in \mathscr{O}(X)$ in a locally compact space $X$.
Then there exists a set  $V \in \mathscr{O}(X)$ such that $ C = \overline V$ is compact and  
$ K \subseteq V \subseteq \overline V \subseteq U. $
If $X$ is also locally connected, and either $K$ or $U$ is connected, then $V$ and $C$ can be chosen to be connected. 
\end{lemma}

\begin{definition}\label{DTM}
A  deficient topological measure on a locally compact space $X$ is a set function
$\nu:  \mathscr{C}(X) \cup \mathscr{O}(X) \longrightarrow [0, \infty]$ 
which is finitely additive on compact sets, inner compact regular, and 
outer regular, i.e. :
\begin{enumerate}[label=(DTM\arabic*),ref=(DTM\arabic*)]
\item \label{DTM1}
if $C \cap K = \emptyset, \ C,K \in \mathscr{K}(X)$ then $\nu(C \sqcup K) = \nu(C) + \nu(K)$; 
\item \label {DTM2} 
$ \nu(U) = \sup\{ \nu(C) : \ C \subseteq U, \ C \in \mathscr{K}(X) \} $
 for $U\in\mathscr{O}(X)$;
\item \label{DTM3} 
$ \nu(F) = \inf\{ \nu(U) : \ F \subseteq U, \ U \in \mathscr{O}(X) \} $  for  $F \in \mathscr{C}(X)$.
\end{enumerate}
\end{definition} 

\noindent
Clearly, for a closed set $F$, $ \nu(F) = \infty$ iff $ \nu(U) = \infty$ for every open set $U$ containing $F$.
If two deficient topological measures agree on compact sets (or on open sets) then they coincide.

\begin{definition}\label{TMLC}
A topological measure on $X$ is a set function
$\mu:  \mathscr{C}(X) \cup \mathscr{O}(X)  \longrightarrow  [0,\infty]$ satisfying the following conditions:
\begin{enumerate}[label=(TM\arabic*),ref=(TM\arabic*)]
\item \label{TM1} 
if $A,B, A \sqcup B \in \mathscr{K}(X) \cup \mathscr{O}(X) $ then
$
\mu(A\sqcup B)=\mu(A)+\mu(B);
$
\item \label{TM2}  
$
\mu(U)=\sup\{\mu(K):K \in \mathscr{K}(X), \  K \subseteq U\}
$ for $U\in\mathscr{O}(X)$;
\item \label{TM3}
$
\mu(F)=\inf\{\mu(U):U \in \mathscr{O}(X), \ F \subseteq U\}
$ for  $F \in \mathscr{C}(X)$.
\end{enumerate}
\end{definition} 

By $ \mathbf{DTM}(X)$ and $ \mathbf{TM}(X)$ we denote, respectively, the collections of all finite deficient topological measures and all  finite topological measures on $X$.

The following two theorems from \cite[Section 4]{Butler:DTMLC} give criteria for a deficient topological measure to be a topological measure or a measure.

\begin{theorem} \label{DTMtoTM}
Let $X$ be compact, and $\nu$ a deficient topological measure. The following are equivalent:
\begin{enumerate}
\item[(a)]
$\nu$ is a real-valued topological measure;
\item[(b)]
$\nu(X) = \nu(C)  + \nu(X \setminus C), \quad C \in \mathscr{C}(X);$ 
\item[(c)]
$\nu(X) \le \nu(C)  + \nu(X \setminus C),\quad C \in \mathscr{C}(X).$  
\end{enumerate}
Let $X$ be locally compact, and $\nu$ a deficient topological measure. 
The following are equivalent:
\begin{enumerate}
\item[(a)]
$\nu$ is a topological measure;
\item[(b)]
$\nu(U) = \nu(C)  + \nu(U \setminus C), \quad C \in \mathscr{K}(X),\quad U \in \mathscr{O}(X);$ 
\item[(c)]
$\nu(U) \le \nu(C)  + \nu(U \setminus C),\quad C \in \mathscr{K}(X),\quad U \in \mathscr{O}(X).$
\end{enumerate}
\end{theorem}

\begin{theorem} \label{subaddit}
Let $\mu$ be a deficient topological measure on a locally compact space $X$. 
The following are equivalent: 
\begin{itemize}
\item[(a)]
If $C, K$ are compact subsets of $X$, then $\mu(C \cup K ) \le \mu(C) + \mu(K)$.
\item[(b)]
If $U, V$ are open subsets of $X$,  then $\mu(U \cup V) \le \mu(U) + \mu(V)$.
\item[(c)]
$\mu$ admits a unique extension to an inner regular on open sets, outer regular Borel measure 
$m$ on the Borel $\sigma$-algebra of subsets of $X$. 
$m$ is a Radon measure iff $\mu$ is compact-finite. 
If $\mu$ is finite then $m$ is an outer regular and inner closed regular Borel measure.
\end{itemize}
\end{theorem}

\begin{remark} \label{Vloz}
Let $X$ be locally compact, and let $ \mathscr{M}$  be the collection of all Borel measures on $X$ that are inner regular on open sets and outer regular 
on all Borel sets. Thus, $  \mathscr{M}$ includes regular Borel measures and Radon measures. 
We denote by $M(X)$ the restrictions to $\mathscr{O}(X) \cup \mathscr{C}(X)$ of measures from $ \mathscr{M}$, and by $ \mathbf{M}(X)$ the set of all finite measures from $M(X)$.
We have:
\begin{align} \label{incluMTD}
 M(X) \subsetneqq  TM(X)  \subsetneqq  DTM(X).
\end{align}
The inclusions follow from the definitions. 
When $X$ is compact, there are examples of topological measures that are not measures 
and of deficient topological measures that are not topological measures in numerous papers, 
beginning with \cite{Aarnes:TheFirstPaper}, \cite{OrjanAlf:CostrPropQlf}, and  \cite {Svistula:Signed}.
When $X$ is locally compact, see \cite{Butler:TechniqLC},
Sections 5 and 6 in \cite{Butler:DTMLC}, and  Section 9 in \cite{Butler:TMLCconstr} 
for more information on proper inclusion in (\ref{incluMTD}), criteria for a deficient topological measure to be a measure from $ M(X)$, and various examples.
\end{remark}

\begin{remark} \label{tausm}
In \cite[Section 3]{Butler:DTMLC} we show that a deficient topological measure $ \nu$ is $\tau$-smooth on compact sets
(i.e. if  a net $K_\alpha \searrow K$ , where $ K_\alpha, K \in \mathscr{K}(X)$ then  $\mu(K_\alpha) \rightarrow \mu(K)$), 
and also $\tau$-smooth on open sets (i.e.  if a net $U_\alpha \nearrow U$, where $U_\alpha, U \in \mathscr{O}(X)$ then $\mu(U_\alpha) \rightarrow \mu(U)$). 
In particular, a deficient topological measure is additive on open sets.
A deficient topological measure $ \nu$ is also superadditive, i.e. 
if $ \bigsqcup_{t \in T} A_t \subseteq A, $  where $A_t, A \in \mathscr{O}(X) \cup \mathscr{C}(X)$,  
and at most one of the closed sets (if there are any) is not compact, then 
$\nu(A) \ge \sum_{t \in T } \nu(A_t)$. 
If $ F \in \mathscr{C}(X)$ and $C \in \mathscr{K}(X)$ are disjoint, then $ \nu(F) + \nu(C) = \nu ( F \sqcup C)$.
One may consult  \cite{Butler:DTMLC} for more properties of deficient topological measures on locally compact spaces.
\end{remark}

\begin{definition} \label{SDTMnorDe}
For  a deficient topological measure  $\mu$ we define $ \| \mu \| = \mu(X) = \sup \{ \mu(K): K \in \mathscr{K}(X) \}$.
\end{definition}

\begin{definition} \label{cqlf}
We call a  functional $\rho$ on $C_0(X)$  with values in $[ -\infty, \infty]$ (assuming at most one of $\infty, - \infty$) 
and $| \rho(0) | < \infty$ a p-conic quasi-linear functional if 
\begin{enumerate}[label=(p\arabic*),ref=(p\arabic*)]
\item
If $f\, g=0, f, g  \ge 0$ then $ \rho(f+ g) = \rho(f) + \rho(g)$.
\item
If $0 \le g \le f$ then $\rho(g) \le \rho(f)$.
\item
For each $f$, if $g,h \in A^+(f), \ a,b \ge 0$ then $\rho(a g + bh) = a \rho(g) + b \rho(h)$.
Here $ A^+(f) = \{ \phi \circ f: \ \phi \in C(\overline{f(X)}), \phi  \mbox{   is non-decreasing}\} $, (with $ \phi(0) = 0 $ 
if $X$ is non-compact) is a cone generated by $f$.
\end{enumerate}

For a functional $\rho$ on $C_0(X)$ we consider $\| \rho \| =  \sup \{ | \rho(f) | : \  \| f \| \le 1 \} $ and we say $\rho$ is bounded if $\| \rho \| < \infty$.
Let $\mathbf{\Phi^+}(C_0^+(X))$ be the set of all  bounded p-conic quasi-linear $\rho$ functionals on $C_0^+(X)$. 

A real-valued map $\rho$ on $C_0(X)$
is a quasi-linear functional (or a positive quasi-linear functional) if 
\begin{enumerate}[label=(QI\arabic*),ref=(QI\arabic*)]
\item \label{QIpositLC}
$ f \ge 0 \Longrightarrow \rho(f)  \ge 0.$
\item \label{QIconsLC}
$\rho(a f) = a \rho(f)$ for $ a \in \mathbb{R}.$
\item \label{QIlinLC}
For each  $f $,   if $g,h \in B(f)$, then $\rho(h + g) =  \rho (h) + \rho (g)$.
Here  $B(f) =  \{ \phi \circ f : \,  \phi \in C(\overline{f(X)}) \} $  (with $ \phi(0) = 0 $ if $X$ is non-compact) is a subalgebra generated by $f$. 
\end{enumerate}
\end{definition}

\begin{remark} \label{RemBRT}

There is an order-preserving bijection between $\mathbf{DTM}(X)$ and $\mathbf{\Phi^+}(C_0^+(X))$. See \cite[Section 8]{Butler:ReprDTM}.
In particular, there is an order-preserving isomorphism between finite 
topological measures on $X$ and quasi-linear functionals on $C_0(X)$ of finite norm, and $\mu$ is a measure iff the 
corresponding functional is linear (see \cite[Theorem 8.7]{Butler:ReprDTM}, \cite[Theorem 3.9]{Alf:ReprTh}, and \cite[Theorem 15]{Svistula:DTM}).
We outline the correspondence.
\begin{enumerate}[label=(\Roman*),ref=(\Roman*)]
\item \label{prt1}
Given a finite deficient  topological measure $\mu$ on a locally compact space $X$ and $f \in C_b(X)$, define functions on $\mathbb{R}$:
$$ R_1 (t) = R_{1, \mu, f} (t) =  \mu(f^{-1} ((t, \infty) )), $$
$$ R_2 (t) =  R_{2,  \mu, f} (t) =\mu(f^{-1} ([t, \infty) )). $$
Let $r$ be the Lebesque-Stieltjes measure associated with $-R_1$, a regular Borel measure on $ \mathbb{R}$. The $ supp \ r \subseteq \overline{f(X)}$.
We define a functional on $C_b(X)$ (in particular, a functional on  $C_0(X)$):
\begin{align} \label{rfform}
\mathcal{R} (f) & = \int _{\mathbb{R}}  id \, dr = \int_{[a,b]} id \, dr  =   \int_a^b R_1 (t) dt + a \mu(X)  =  \int_a^b R_2 (t) dt + a \mu(X).  
\end{align}
where $[a,b]$ is any interval containing $f(X)$.
If $f(X) \subseteq [0,b]$ we have:
\begin{align*} 
 \mathcal{R} (f) = \int_{[0,b]}  id \, dr  =   \int_0^b R_1 (t) dt =   \int_0^b R_2 (t) dt.
\end{align*}
We call the functional $\mathcal{R}$ a quasi-integral (with respect to a deficient topological measure $ \mu$) and write:
\begin{align*} 
\int_X f \, d\mu = \mathcal{R}(f) = \mathcal{R}_{\mu} (f) =  \int _{\mathbb{R}}  id \, dr.
\end{align*}
\item   \label{RHOsvva}
Functional $\mathcal{R} $ is non-linear. 
By   \cite[Lemma 7.7,  Theorem 7.10, Lemma 3.6, Lemma 7.12]{Butler:ReprDTM}  we have:
\begin{enumerate}
\item
$\mathcal{R} (f) $ is positive-homogeneous, i.e. $\mathcal{R} (cf)  = c \mathcal{R} (f) $ for $c \ge 0$ and $ f \in C_b(X)$. 
\item
$\mathcal{R} (0) =0$. 
\item 
$\mathcal{R}$ is monotone, i.e. if $ f \le g$  then $\mathcal{R} (f) \le \mathcal{R} (g)$ for $f, g \in C_b(X)$.
\item
$ \mu(X)  \cdot \inf_{x \in X} f(x)  \le \mathcal{R}(f)  \le \mu(X) \cdot \sup _{x \in X} f(x) $ for $f \in C_b(X)$. 
\item
If $f g = 0 $, where $ f, g \ge 0$  then $\mathcal{R} (f+g) = \mathcal{R} (f) + \mathcal{R} (g)$ for $f, g \in C_b(X)$; \\
if $f g = 0 $, where $f \ge 0, g \le 0$ or $ f, g \ge 0$, then $\mathcal{R} (f+g) = \mathcal{R} (f) + \mathcal{R} (g)$ for $f, g \in C_0(X)$.
\end{enumerate}
\item \label{mrDTM} 
A functional $\rho$ with values in $[ -\infty, \infty]$ (assuming at most one of $\infty, - \infty$) and $| \rho(0) | < \infty$ 
is called a d-functional if   
on nonnegative functions it is positive-homogeneous, monotone, and orthogonally additive, i.e. for $f, g \in D(\rho)$ (the domain of $ \rho$) we have: 
(d1) $f \ge 0, \ a > 0  \Longrightarrow  \rho (a f) = a \rho(f)$; 
(d2) $0 \le  g \le f \Longrightarrow  \rho(g) \le \rho(f) $;
(d3) $f \cdot g = 0, f,g \ge 0  \Longrightarrow  \rho(f + g) = \rho(f) + \rho(g)$. 

Let  $\rho$ be a d-functional with $  C_c^+(X) \subseteq D(\rho) \subseteq C_b(X)$. 
In particular, we may take functional $ \mathcal{R}$ on $  C_0^+(X)$. 
The corresponding
deficient topological measure $ \mu = \mu_{\rho}$ is given as follows: 

If $U$ is open, $ \mu_{\rho}(U) = \sup\{ \rho(f): \  f \in C_c(X), 0\le f \le 1,  supp \, f\subseteq U  \},$

if $F$ is closed, $ \mu_{\rho}(F) = \inf \{ \mu_{\rho}(U): \  F \subseteq U,  U \in \mathscr{O}(X) \}$. 

If $K$ is compact, $ \mu_{\rho}(K) = \inf \{ \rho(g): \   g \in C_c(X), g \ge 1_K \}  
= \inf \{ \rho(g): \   g \in C_c(X), 1_K \le g \le 1 \}. $
(See \cite[Section 5]{Butler:ReprDTM}.)
\end{enumerate}

If given a finite deficient topological measure $\mu$, we obtain $ \mathcal R$, and then $\mu_{ \mathcal R}$, then $ \mu = \mu_{ \mathcal R}$.
\end{remark}

\begin{remark} \label{LipQLF}
Integrals with respect to (deficient) topological measures on a locally compact space $X$ have Lipschitz property:
If $\mu$ is a finite deficient topological measure, $f, g \in C_c(X), \, f,g \ge 0, \, supp \, f, supp \, g \subseteq K$ where $K$ is compact, then 
$$ | \mathcal{R}(f) - \mathcal{R}(g) |  = |  \int_X f \, d\mu - int_X g \, d\mu | \le \| f - g \| \, \mu(K).$$ 
If $ \mu$ is a finite topological measure,  $f, g \in C_0(X)$ then
$$|  \int_X f \, d\mu - int_X g \, d\mu | \le  2 \| f - g \| \, \mu(X).$$
See \cite[Lemma 7.12]{Butler:ReprDTM}  and  \cite[Corollary 53]{Butler:DTMLC}.
\end{remark}

We would like to give some examples. 

\begin{definition}
A set $A$ is bounded if $\overline{A}$ is compact. 
If $X$ is locally compact, non-compact, a set $A$ is solid if $A$ is  connected, and $X \setminus A$ has only unbounded connected components.
If $X$ is compact, a set $A$ is solid if $A$ and $X \setminus A$ are connected.
\end{definition}

Many examples of topological measures that are not measures are obtained in the following way. Define a so-called solid-set function on 
bounded open solid and compact solid sets in a locally compact, connected, locally connected, Hausdorff space. 
A solid set function extends to a unique topological measure. 
See \cite[Definition 2.3, Theorem 5.1]{Aarnes:ConstructionPaper}, \cite[Definition 39, Theorem 48]{Butler:TMLCconstr}.

\begin{example} \label{ExDan2pt}
Suppose that  $ \lambda$ is the Lebesgue measure on $X = \mathbb{R}^2$,  and the set $P$ consists of two points
$p_1 = (0,0)$ and $p_2 = (2,0)$.
For each bounded open solid or compact solid set $A$ let $ \nu(A) = 0$ if $A \cap P = \emptyset$,   
$ \nu(A) = \lambda(A) $ if $A$ contains one point from $P$, and 
$ \nu(A) = 2 \lambda(X)$ if $A$ contains both points from $P$.
Then $\nu$ is a solid-set function (see \cite[Example 61]{Butler:TMLCconstr}), and  $\nu$ extends to a unique topological measure on $X$. 
Let $K_i$ be the closed ball of radius $1$ centered at $p_i$ for $i=1,2$. Then 
$K_1, K_2$ and $ C= K_1 \cup K_2$ are compact solid sets, $\nu(K_1) = \nu(K_2) = \pi, \,  \nu(C) = 4 \pi$. Since 
$\nu$ is not subadditive, it can not be a measure.  The quasi-linear functional corresponding to $ \nu$ is not linear. 
\end{example}

\begin{example} \label{nvssf}
Let  $X = \mathbb{R}^2$ or a square, $n$ be a natural number, and let $P$ be a set of distinct $2n+1$ points.
For each bounded open solid or compact solid set $A$ let $ \nu(A) = i/n$ if $ A$ contains  $2i$ or $2i+1$ points from $P$.
The set function $ \nu$  defined in this way is a solid-set function, and it extends to a unique topological measure on $X$ 
that assumes values $0, 1/n, \ldots, 1$. 
See  \cite[Example 2.1]{Aarnes:Pure},   \cite[Examples 4.14, 4.15]{QfunctionsEtm}, and \cite[Example 65]{Butler:TMLCconstr}.
The resulting topological measure is not a measure. For instance, when $X$ is the square and $n=3$, it is easy to represent $X = A_1 \cup A_2 \cup A_3$,
where each $A_i$  is a compact solid set containing one point from $P$. Then $\nu(A_i) =0$ for $i=1,2,3$, while $\nu(X) = 1$.  
Since $\nu$ is not subbadditive, it is not a measure, and the quasi-linear functional  $\rho$ corresponding to $ \nu$ is not linear. 
In \cite[Example 56]{Butler:QLFLC} we take $n=5$ and show that there are $f,g \ge 0$ such that $ \rho(f+g) \neq \rho(f) + \rho(g)$.
If $X$ is locally compact, non-compact, for the functional $\rho$ we consider
a new functional $ \rho_g$ defined by $\rho_g(f) = \rho(gf)$, where $g \ge 0$. 
The new functional $\rho_g$ corresponds to a deficient topological measure 
obtained by integrating $g$ over closed and open sets with respect to a topological measure $\nu$. We can choose $ g \ge 0$ or $ g >0$ so that  
$\rho_g$ is no longer linear on singly generated subalgebras, but only linear on singly generated cones. 
See \cite[Example 32, Theorem 40]{Butler:Integration} for details.
\end{example}

\begin{example}
Let $X$ be locally compact, and let $D$ be a connected compact subset of $X$. Define a set function 
$\nu$ on $\mathscr{O}(X) \cup \mathscr{C}(X)$  by setting $\nu(A) = 1$ if $ D \subseteq A$ and $\nu(A) = 0$ otherwise, for any
$A \in \mathscr{O}(X) \cup \mathscr{C}(X)$. If $D$ has more than one element, then $\nu$ is a deficient topological measure, but not 
a topological measure. See \cite[Example 46]{Butler:DTMLC} and \cite[Example 1, p.729]{Svistula:DTM} for details.
\end{example}

For  more examples of topological measures and quasi-integrals on locally compact spaces
see \cite{Butler:TechniqLC} and the last sections of \cite{Butler:TMLCconstr} and \cite{Butler:QLFLC}.
For more examples of deficient topological measures see  \cite{Butler:DTMLC} and \cite{Svistula:DTM}. 

\section{Aleksandrov's Theorem for deficient topological measures}

\begin{definition} \label{defwk}
The weak topology on $ \mathbf{DTM}(X)$ is the coarsest (weakest) topology for which maps $ \mu \longmapsto \mathcal{R}_{\mu} (f), f \in C_0^+(X) $ are continuous.
\end{definition}

The basic neighborhoods for the weak topology have the form 
\begin{align} \label{wstRho}
N(\nu, f_1, \ldots, f_n, \epsilon) = \{ \mu \in \mathbf{DTM}(X): \ |\mathcal{R}_{\mu}(f_i) - \mathcal{R}_{\nu} (f_i) | &< \epsilon, \, f_i \in C_0^+(X), \\
&i=1, \ldots, n  \notag \}.
\end{align} 

Let $\mu_{\alpha}$ be a net in $\mathbf{DTM}(X)$, $\mu \in \mathbf{DTM}(X)$. 
The net $ \mu_{\alpha} $ converges weakly to $ \mu$ (and we write $ \mu_{\alpha} \Longrightarrow \mu$)  
iff  $\mathcal{R}_{\mu_{\alpha}}  (f) \rightarrow  \mathcal{R}_{\mu}  (f)$  for every $ f \in C_0^+(X)$, i.e. 
$\int f \, d \mu_{\alpha} \rightarrow \int f \, d \mu$  for every $ f \in C_0^+(X)$.

By  \cite[Theorem 8.7]{Butler:ReprDTM}, $\mathbf{DTM}(X)$ with weak convergence is homeomorphic to  $\mathbf{\Phi^+}(C_0^+(X))$ 
with pointwise convergence, 
and $\mathbf{TM}(X)$ is homeomorphic to the space of quasi-linear functionals with pointwise convergence.

\begin{remark} \label{wkwk*}
Our definition of weak convergence corresponds to one used in probability theory. It is the same as a functional analytical definition of  $wk*$ convergence
on $\mathbf{DTM}(X)$ (respectively, on $\mathbf{TM}(X)$), which is justified by the fact that this topology agrees with the weak$^*$ topology 
induced by p-conic quasi-linear functionals (respectively, quasi-linear functionals). 
In many papers the term "$wk^*$-topology" is used. 
\end{remark}

\begin{definition} \label{muContSet}
Let $\mu$ be a deficient topological measure.
A set $A$ is called a $\mu$-continuity set if  $ \mu (\overline A) = \mu(A^o)$.
\end{definition}

\begin{remark}
In probability theory, with $\mu$ a measure, a set $A$ is called a  $\mu$-continuity set if $\mu( \partial A) =0$. 
If $\mu$ is a measure (or $\mu$ is a topological measure and $\overline A$ is compact) this definition is equivalent to Definition \ref{muContSet}. 
If $ \mu$ is a deficient topological measure, then
by superadditivity $ \mu( \overline A )\ge  \mu(A^o) + \mu( \partial A)$, so for any  $\mu$-continuity set  $A$ we have $\mu( \partial A) =0$.
\end{remark}

We have the following generalizations of Aleksandrov's  well-known theorem for weak convergence of measures. 
(Aleksandrov's Theorem is often incorrectly called the "Portmanteau theorem",
a usage apparently deliberately started by Billingsley, who in \cite{BillingsleySv} cited a paper of the non-existent mathematician Jean-Pierre Portmanteau, 
"published" in a non-existent issue of the Annals of non-existent university;  see \cite[p.130]{Pfanzagl} and \cite[p.313]{Simon}.) 
This theorem gives equivalent definitions of weak convergence.

\begin{theorem} \label{AleksandrovLC} 
Let $X$ be locally compact,  and let $\mu, \mu_{\alpha}$  be deficient topological measures. The following are equivalent: 
\begin{enumerate}[label=(\arabic*),ref=(\arabic*)] 
\item \label{portm1}
$\int f \, d \mu_{\alpha} \rightarrow \int f \, d \mu$ ( i.e. $\mathcal{R}_{\mu_{\alpha}}  (f) \rightarrow  \mathcal{R}_{\mu}  (f)$ )  for every $ f \in C_0^+(X)$.
\item \label{portm3}
$\liminf \mu_{\alpha} (U)  \ge \mu(U)$ for any $ U \in \mathscr{O}(X)$ and $ \limsup \mu_{\alpha} (K) \le \mu(K)$ for any $ K \in \mathscr{K}(X)$.
\item \label{portm4}
$ \mu_{\alpha} (A) \rightarrow \mu(A)$ for any compact or open bounded $\mu$-continuity set $A$.
\item \label{portm5} 
If $ f \in C_0^+(X)$  then $R_{2, \mu_{\alpha}, f} (t)  \rightarrow R_{2, \mu, f} (t) $  and $R_{1, \mu_{\alpha}, f} (t)  \rightarrow R_{1, \mu, f} (t) $ 
for each point $t$ at which $R_{2, \mu, f}$ is continuous.  
\end{enumerate}
\end{theorem}  

\begin{proof}
\ref{portm1} $ \Rightarrow $  \ref{portm3}.  
Let $ U \in \mathscr{O}(X), \ K \in \mathscr{K}(X), \epsilon >0$.
By part \ref{mrDTM} of Remark \ref{RemBRT}
choose $f, g \in C_c(X)$ such that $supp f \subseteq U, \ K \subseteq supp\, g$ and
$ \mathcal{R}_{\mu} (f) > \nu(U) - \epsilon, \ \ \mathcal{R}_{\mu} (g) < \nu(K) +  \epsilon.$
Choose $ \alpha_0$ such that $ | \mathcal{R}_{\mu_{\alpha}} (f) - \mathcal{R}_\mu (f) |  < \epsilon $ and  
$ | \mathcal{R}_{\mu_{\alpha}} (g) - \mathcal{R}_\mu (g) |  < \epsilon $  for all $ \alpha > \alpha_0$.
Then 
\begin{align*} 
\mu_{\alpha} (U) \ge \mathcal{R}_{\mu_{\alpha}} (f) > \mathcal{R}_{\mu} (f) - \epsilon >  \mu(U) -  2\epsilon , \\
\mu_{\alpha} (K)  \le \mathcal{R}_{\mu_{\alpha}} (g) <   \mathcal{R}_{\mu}(g) + \epsilon < \mu(K) +  2 \epsilon,
\end{align*}
and it is easy to see that $\liminf \mu_{\alpha} (U)  \ge \mu(U)$ and $\limsup  \mu_{\alpha} (K) \le \mu(K)$.  \\
\ref{portm3} $ \Rightarrow $  \ref{portm4}.   We have:
$\mu(A^o)  \le \liminf  \mu_{\alpha} (A^o)  \le  \liminf  \mu_{\alpha} (A) \le   \limsup  \mu_{\alpha} (A) \le  \limsup  \mu_{\alpha} (\overline A) \le \mu(\overline A). $
If $A$ is an $\mu$-continuity set (whether $A$ is compact or open bounded), we then see that $\lim  \mu_{\alpha} (A) = \mu(A)$. \\
\ref{portm4} $ \Rightarrow $ \ref{portm5}.  
If $ t$ is a point of continuity of $ R_{2, \mu, f}$ then from \cite[Lemma 6.3 (III)]{Butler:ReprDTM} it follows that 
the sets $ f^{-1}((t, \infty))$ and $ f^{-1}([t, \infty))$ are $\mu$-continuity sets. The statement follows from \ref{portm4}.  \\
\ref{portm5} $ \Rightarrow $ \ref{portm1}. 
By  \cite[Lemma 6.3]{Butler:ReprDTM} $R_{2, \mu, f}$ has at most countably many points of discontinuity; the statement follows from formulas 
(\ref{rfform}) and \ref{portm5}.
\end{proof}  

If $\mu, \mu_{\alpha}$   are finite topological measures on a compact space $X$, and $ \lim \mu_{\alpha} (X) = \mu(X)$,  
then from part \ref{TM1} of Definition \ref{TMLC} it follows that 
$\liminf \mu_{\alpha} (U) \ge \mu(U)$ for any $ U \in \mathscr{O}(X)$  iff $ \limsup \mu_{\alpha} (D) \le \mu(D)$ for any $ D \in \mathscr{C}(X)$.
Therefore, we have the following version of Aleksandrov's  Theorem:
 
\begin{theorem} \label{AleksandrovLCtm} 
Let $X$ be compact,  and let $\mu, \mu_{\alpha}$  be  finite topological measures. TFAE: 
\begin{enumerate}[label=(\arabic*),ref=(\arabic*)] 
\item \label{portm1tm}
$\int f \, d \mu_{\alpha} \rightarrow \int f \, d\mu$ ( i.e. $\mathcal{R}_{\mu_{\alpha}}  (f) \rightarrow  \mathcal{R}_{\mu}  (f)$ )  for every $ f \in C(X)$.
\item \label{portm3tm}
$\liminf \mu_{\alpha} \ge \mu(U)$ for any $ U \in \mathscr{O}(X)$ and $ \lim \mu_{\alpha} (X) = \mu(X)$. 
\item \label{portm3atm}
$ \limsup \mu_{\alpha} (D) \le \mu(D)$ for any $ D \in \mathscr{C}(X)$ and $ \lim \mu_{\alpha} (X) = \mu(X)$.
\item \label{portm4tm}
$ \mu_{\alpha} (A) \rightarrow \mu(A)$ for  any $\mu$-continuity set $A$.
\item \label{portm5tm} 
If $ f \in C_0^+(X)$  then $R_{2, \mu_{\alpha}, f} (t)  \rightarrow R_{2, \mu, f} (t) $ and $R_{1, \mu_{\alpha}, f} (t)  \rightarrow R_{1, \mu, f} (t) $  
for each point $t$ at which $R_{2, \mu, f}$ is continuous.
\end{enumerate}
\end{theorem} 

\begin{theorem} \label{wkBase}
The weak topology on $\mathbf{DTM}(X)$ is given by basic neighborhoods of the form 
\begin{align*}   
 W( &\nu, U_1, \ldots, U_n, C_1, \ldots, C_m, \epsilon) = \{ \mu \in DTM: \ \mu(U_i) > \nu(U_i) - \epsilon, \ \mu(C_j) < \nu(C_j) + \epsilon,  \\
& i=1, \ldots, n, \ j=1, \ldots m \} 
\end{align*} 
where $\nu \in \mathbf{DTM}(X),  U_i \in \mathscr{O}(X), C_j \in \mathscr{K}(X), \epsilon >0, n, m \in \mathbb{N}$.
\end{theorem}

\begin{proof}
The weak topology is the topology $\tau_N$ given by basic neighborhoods of the form (\ref{wstRho}). 
It is easy to see that the sets $ W(\nu, U_1, \ldots, U_n, C_1, \ldots, C_m, \epsilon)$ are basic neighborhoods for some topology $\tau_W$ on  $\mathbf{DTM}(X)$.
Consider a basic neighborhood $ W(\nu, U, C, \epsilon)$.
Given $\epsilon >0$, by part \ref{mrDTM} of Remark \ref{RemBRT}
choose $f, g \in C_c(X)$ such that  $ supp f \subseteq U, g \ge 1_K$ and 
$$ \mathcal{R}_{\nu} (f) > \nu(U) - \frac{\epsilon}{2}, \ \ \mathcal{R}_{\nu} (g) < \nu(C) +  \frac{\epsilon}{2}.$$ 
Let $ \mu \in N(\nu, f,g, \epsilon/2)$  as in (\ref{wstRho}).  We have:
$$ \mu(U)  > \mathcal{R}_{\mu} (f) > \mathcal{R}_{\nu} (f) - \frac{\epsilon}{2} >  \nu(U) - \epsilon, $$
$$  \mu(C) \le \mathcal{R}_{\mu}(g) <  \mathcal{R}_{\nu}(g) + \frac{\epsilon}{2}  <  \nu(C) +  \epsilon.$$  
Therefore, $  N(\nu, f, g, \epsilon/2) \subseteq W(\nu, U, C, \epsilon)$.
We see that $\tau_W \subseteq \tau_N$, i.e. 
$\tau_W$ is a coarser topology than $\tau_N$.
If $ \mu_{\alpha} \rightarrow \mu$ in the topology $\tau_W$ then it is easy to see that $\liminf  \mu_{\alpha}(U) \ge \mu(U)$ for any open set $U$, and that
$\limsup \mu_{\alpha}(K) \ge \mu(K)$ for any compact set $K$. By Theorem \ref{AleksandrovLC} $\int f \, d \mu_{\alpha} \rightarrow \int f \, d\mu$ for every 
$f \in C_0^+(X)$. The weak topology $\tau_N$ is the coarsest topology with this property, thus, $ \tau_N = \tau_W$.
\end{proof}
 
\begin{theorem}  \label{basicTP}
The space $\mathbf{DTM}(X)$ is Hausdorff and locally convex. Every set of the form 
$\{ \mu \in \mathbf{DTM}(X): \mu(X) \le c\} = \{ \mathcal{R}: \| \mathcal{R} \|  \le c\} , c>0$ is compact. 
If $X$ is compact then   $\mathbf{DTM}(X)$ is locally compact.
\end{theorem}

\begin{proof}
First we shall show that  $\mathbf{DTM}(X)$ is Hausdorff. 
Suppose $ \mu \neq \nu$, then there is $ K \in \mathscr{K}(X)$ such that $ \nu(K) \neq \mu(K)$. Let $ | \mu(K) - \nu(K) | = 5 \epsilon >0$.
By part \ref{mrDTM} of Remark \ref{RemBRT}
find $g, h \in C_c(X)$ such that $ \mathcal{R}_{\mu}(g) - \mu(K) < \epsilon, \, \mathcal{R}_{\nu}(h) - \nu(K) < \epsilon$. Let $f = g \wedge h$, so 
 $ \mathcal{R}_{\mu}(f) - \mu(K) < \epsilon, \, \mathcal{R}_{\nu}(f) - \nu(K) < \epsilon$.
Then $N(\mu, f, \epsilon)$ and $N(\nu, f , \epsilon)$ as in formula (\ref{wstRho}) are disjoint 
neighborhoods of $ \mu$ and $ \nu$: otherwise, if $ \lambda \in N(\mu, f, \epsilon) \cap N(\nu, f, \epsilon) $ then 
$ | \mu(K) - \nu(K) | \le | \mu(K) - \mathcal{R}_{\mu}(f)| + | \mathcal{R}_{\mu}(f) - \mathcal{R}_{\lambda} (f) | 
+ |  \mathcal{R}_{\lambda}(f) - \mathcal{R}_{\nu} (f) | + | \mathcal{R}_{\nu} (f) - \nu(K)| < 4\epsilon  
<  | \mu(K) - \nu(K) | $, which is a  contradiction.

One can also see that  $\mathbf{DTM}(X)$ is Hausdorff because a homeomorphic space $\mathbf{\Phi^+}(C_0^+(X))$ is  Hausdorff.
The basic open set in $\mathbf{\Phi^+}(C_0^+(X))$  is of the form $W = \{ \mathcal{R}: \mathcal{R}(f_i) \in O_i, \, O_ i 
\mbox{   are open in   } \mathbb{R},  \, f_i \in C_0^+(X), \, i=1, \ldots, n,  \}$.  
If $ \mathcal{R} $ and $\rho$ are in $W$, then their convex combination is also in $W$. Thus, $\mathbf{DTM}(X)$ is locally convex. 

Let $  c>0$ and $ P= \{ \mu \in \mathbf{DTM}(X): \mu(X) \le c\}$. 
Consider the product space
\begin{eqnarray*}
Y = \prod_{f \in C_0^+(X)} [ -c \| f \| , c \| f \| \ ]  
\end{eqnarray*} 
and the function
$ T : P \longrightarrow Y$   defined by   $( T(\mu) )_f = \rho_{\mu} (f) = \int f \, d\mu $.
The function $T$ is continuous , since each of the maps
$ \mu \longmapsto \rho_{\mu} (f) $ is continuous.
$ T$ is $ 1-1 $ which follows from Remark \ref{RemBRT}.
Also $ T : P \longrightarrow T(P) $ is a 
homeomorphism, because 
$ T(\mu_{\gamma}) \longrightarrow T(\mu_0) $
implies $ \mu_{\gamma} \longrightarrow \mu_0 $.
To show that $P$ is compact it is enough to show that
$ T(P) $ is closed in $Y$.
Let $ T(\mu_{\alpha}) \longrightarrow L $ in $ Y $.
Define $ \rho(f) = L_f, \,  f \in C_0^+(X) $.
Then $\rho$ is a p-conic quasi-linear functional, and by Remark \ref{RemBRT} there exists a finite deficient topological measure 
$ \mu_0 $ such that
$ \rho = \rho_{\mu_0} $.
Then $ L_f= \rho_f=(\rho_{\mu_0})_f = (T(\mu_0))_f $,
i.e. $ L=T(\mu_0) $.

If $X$ is compact, then for $\nu \in \mathbf{DTM}(X)$  and $W = \{ \mu : \mu(X) < \nu(X) + \epsilon\}$ we have: 
$ \nu \in W \subseteq \{ \mu: \mu(X) \le \nu(X) + \epsilon \}$, and the last set is compact.
\end{proof}

\section{Prokhorov's Theorem for topological measures}

In this section we show that several classical results of probability theory hold for deficient topological measures or topological measures. 
 
\begin{lemma}
If each sequence $\{\mu_{n_i}\} $ of   $\{\mu_{n}\} $, where $\mu_n$ are deficient topological measures, contains a further subsequence
 $\{\mu_{n_{i_j}}\} $ such that $\mu_{n_{i_j}}$ converges weakly to a deficient topological measure $\mu$, then $\mu_n$ converges weakly to  $\mu$.
 \end{lemma}
 
\begin{proof}
If   $\mu_n$  does not converge weakly to  $\mu$, then there is $f \in C_0^+(X)$ such that $| \int f \, d \mu_{n_i} - \int f \, d\mu | \ge  \epsilon$ for some $ \epsilon>0$ and 
all $\mu_{n_i}$ in some subsequence. But then no subsequence of $ \{ \mu_{n_i} \}$ can converge weakly to $\mu$.
\end{proof}  

We clearly have 

\begin{lemma} \label{L6.1}
$X$ is homeomorphic to the (topological) subset $D = \{ \delta_x: \, x \in X\}$ of $ \mathbf{DTM}(X)$ (equipped with the weak topology).
\end{lemma}


\begin{theorem} \label{metrization}
Let $c\ge 0$. Then $ P = \{ \mu \in \mathbf{DTM}(X): \mu(X) \le c\} $ can be metrized as a separable metric space iff $X$ is a separable metric space.
\end{theorem}

\begin{proof}
Suppose $X$ is a separable metric space. By Urysohn's metrization theorem (see \cite[p.125]{Kelly}) $X$ can be topologically embedded in a countable product of 
unit intervals. Consequently, there exists an equivalent totally bounded metrization on $X$. We will consider this metric on $X$. 
From  \cite[Lemma 6.3]{Parthasarathy}   $C_b(X)$ is separable. Let $\{ f_1, f_2, \ldots \}$  be a countable dense subset of  $C_b(X)$. 

Let $Y$ be a countable product of $\mathbb{R}$. Define a map $T : P \longrightarrow Y$ as in Theorem \ref{basicTP}, i.e.
$T(\mu) = (\int f_1 \, d \mu, \int f_2 \, d \mu, \ldots)$. We will show that $T$ is a homeomorphism on $P$.
First, $T$ is $1-1$. (If $T(\mu) = T(\nu)$ then $\int f_i \, d \mu = \int f_i \, d \nu$ for all $i$, and, hence,  
$\int f \, d \mu = \int f \, d \nu$ for all $ f \in C_0^+(X)$. 
By Remark \ref{RemBRT}, $ \mu = \nu$.) Second, $T$ and $T^{-1}$  are continuous, as in the proof of Theorem \ref{basicTP}. 
Since   $Y$ is a separable metric space, and $P$ is homeomorphic to a subset of  $Y$, it follows that $P$ is a separable metric space. 

Conversely, suppose $P$ is a separable metric space. By Lemma \ref{L6.1} $X$ is homeomorphic to $ D =\{ \delta_x: \, x \in X\}$. $D$ is 
a separable metric space, and then so is $X$.
\end{proof}

\begin{definition} \label{UnTight}
Let $X$ be locally compact.
A family $\mathcal{M}  \subseteq \mathbf{DTM}(X)$ is  uniformly tight if for every $ \epsilon >0$ there exists a compact set $ K_\epsilon$ 
such that $\mu(K_\epsilon) > \epsilon $ for each  $\mu \in \mathcal{M}$.
A family $\mathcal{M}  \subseteq \mathbf{DTM}(X)$ is  uniformly bounded in variation if there is a positive constant $M$ such that $ \| \mu \| \le M$ 
for each $\mu \in \mathcal{M} $. 
\end{definition}

One uniformly bounded in variation family that is the often used is the collection of all normalized (i.e. satisfying condition $\mu(X) = 1$) topological measures on a compact space. 

\begin{proposition} \label{wkseqBd}
Suppose $X$ is locally compact.
If a sequence $(\mu_n) \in \mathbf{DTM}(X)$ is weakly fundamental (i.e. $ \int f \, d\mu_n$ is a fundamental sequence for each $ f \in C_0^+(X)$)  
then it is uniformly  bounded in variation.
\end{proposition}

\begin{proof}
If not, then there is a subsequence $ (\mu_{n_k})$ such that $ \| \mu_{n_k} \| > k 2^k$ for each $k$; 
and by 
part \ref{mrDTM} of Remark \ref{RemBRT} there are functions $f_{n_k} \in C_c(X), 0 \le f_{n_k} \le 1$ such that
$ \int_X  f_{n_k}  \, d\mu_{n_k}   > k 2^k$. Then the function $ f = \sum_{k=1}^\infty \frac{f_{n_k} }{2^k} \in C_0^+(X), \, 0 \le f \le 1$, 
and $ \int_X f  \, d\mu_{n_k} \ge k$ for each $k$. This contradicts the fact that the sequence $( \int f \, d\mu_n)$ is Cauchy, hence, bounded.
\end{proof}

\begin{theorem} \label{wkfamBd}
Suppose $\mathcal{M}  \subseteq \mathbf{DTM}(X)$ is a family of finite deficient topological measures such that every sequence in $\mathcal{M} $ 
contains a weakly convergent subsequence. Then  $\mathcal{M}$ is uniformly  bounded in variation.
\end{theorem}

\begin{proof}
If not, then there is a sequence $\mu_n \subseteq \mathcal{M} $ such that $ \| \mu_n \| > n $ for every natural $n$.
Let $ m_{n_k}$ be its weakly convergent subsequence. Then $\|  m_{n_k} \| > n_k$, while by Proposition \ref{wkseqBd} this subsequence must be
uniformly  bounded in variation. 
\end{proof} 
  
\begin{theorem} \label{wkfamTgt}
Suppose $X$ is locally compact. 
Suppose $\mathcal{M}   \subseteq \mathbf{TM}(X)$ is a family of finite topological measures such that every sequence in $\mathcal{M} $ 
contains a weakly convergent subsequence. Then $\mathcal{M} $ is uniformly tight. 
\end{theorem}

\begin{proof}
Suppose $\mathcal{M} $ is not uniformly tight. Then there exists $ \epsilon>0$ such that for every compact $K$ one can find $\mu^K \in \mathcal{M}$ 
with 
\begin{align} \label{vMuKe}
\mu^K (X \setminus K) > \epsilon.
\end{align}
Take $\mu_1$ to be any topological measure with $ \| \mu_1 \|  > \epsilon$, and let $ K_1 \in \mathscr{K}(X)$ be such that $ \mu(K_1) > \epsilon$. 
Then by Lemma \ref{easyLeLC} there is $V_1 \in \mathscr{O}(X) $ with compact closure such that $ K_1 \subseteq V_1$ and  so $ \mu_1(\overline{V_1}) > \epsilon$. 
By (\ref{vMuKe}) find $ \mu_2$ satisfying $\mu_2(X \setminus \overline{V_1}) >\epsilon$, and let $K_2 \in \mathscr{K}(X)$ 
be such that  $ K_2 \subseteq X \setminus \overline{V_1} $ and $ \mu(K_2) > \epsilon$. 
Find $V_2 \in \mathscr{O}(X)$ with compact closure such that 
$K_2 \subseteq V_2 \subseteq \overline{V_2} \subseteq X \setminus \overline{V_1} $,  so  $ \mu_2(\overline{V_2}) > \epsilon$. 
Find a topological measure $\mu_3$ with $\mu_3( X \setminus (\overline{V_1} \sqcup \overline{V_2}) > \epsilon$, and so on.
By induction we find a sequence of compact sets $ K_j$, a sequence of open sets $V_j$ with compact closure, 
and a sequence of topological measures $\mu_j \in \mathcal{M} $
with the following properties: $ K_j \subseteq V_j \subseteq \overline{V_j}$, $ \overline{V_j}$ are pairwise disjoint, and 
$$  \quad  \mu_j(\overline{V_j}) \ge \mu_j(K_j) > 
\epsilon, \quad  \quad K_{j+1} \subseteq \overline{V_{j+1}} \subseteq X \setminus \bigsqcup_{i=1}^j \overline{V_i}.$$
By part \ref{mrDTM} of Remark \ref{RemBRT} find functions $f_j \in C_c(X), 1_K \le f_j \le 1, supp f_j \subseteq V_j, $ 
with  $ \int_X f_j \, d \mu_j > \epsilon$.
By our assumption the sequence $ (\mu_j) $ contains a weakly convergent subsequence.
For notational simplicity, assume that  $ (\mu_j) $ is weakly convergent. 

By Lemma \ref{wkfamBd} we may assume that $\mathcal{M}$ is uniformly bounded in variation by $M$.
 We let 
$$ a_n^i = \int_X f_i \, d \mu_n. $$
Then $a_n:= (a_n^1, a_n^2, \ldots,)$ belongs to $l^1$, because  for each $m \in \mathbb{N}$,
$f_1 \cdot f_2 \cdot \ldots \cdot f_m = 0, f_1 + \ldots + f_m \in C_c(X), 0 \le  f_1 + \ldots + f_m \le 1$,
and so by part \ref{mrDTM} of Remark \ref{RemBRT}
each partial sum $\sum_{i=1}^m a_n^i = \int_X (f_1 + f_2 + \ldots f_m) \, d \mu_n \le \|\mu_n \| \le M$.
With
$$b_n =\sum_{i=1}^\infty  \int_X f_i \, d \mu_n = \| a_n \|_1 \le M,$$
the sequence $( b_n)$ is bounded, and we may chose a convergent  subsequence. To simplify notations, 
we assume that $(b_n)$ itself converges.

Let  $\lambda = (\lambda_i) \in l^\infty$. 
Since $ |\langle \lambda, a_n \rangle | \le \| \lambda \|_{\infty} \, \| a_{n} \|_{1} \le  \| \lambda \|  M,$ 
we see that the sequence of inner products  $\langle \lambda, a_n \rangle$ is bounded, hence, contains a convergent subsequence. 
Again, for notational simplicity
we assume the sequence itself converges. 

By \cite[Lemma 1.3.7]{Bogachev:WkConv} the sequence $(a_n)$ converges in $l_1-$norm. 
Then $\lim_{n \rightarrow \infty} a_n^n = 0$, which contradicts our choice of $f_n$.
\end{proof}
 
\begin{lemma} \label{wkfuncon}
Let $X$ be locally compact. If $(\mu_n)$ is a weakly fundamental sequence of finite deficient topological measures which is also 
uniformly bounded in variation, then $ \mu_n$ converges weakly to some finite deficient topological measure $ \mu$.
\end{lemma} 

\begin{proof}
Consider functional $L$ on $ C_0^+(X)$ defined as $L(f) = \lim_n \int_X f \, d\mu_n$. It is easy to check that $L$ is a p-conic quasi-linear functional. 
Say, $(\mu_n)$ is uniformly bounded in variation by $M$. Since $ L(f) \le \| \mu_n\|  \le M$ for any $ f\in C_0^+(X), 0 \le f \le1$, 
we see that $L \in \mathbf{\Phi^+}(C_0^+(X))$, and
by Remark \ref{RemBRT} there is a finite deficient topological measure $ \mu$ such that $L(f) = \int_X f \, d \mu$.
\end{proof} 

\begin{theorem} \label{EbSmul}
Suppose $X$ is a locally compact space such that $C_0^+(X)$ is separable. 
Then every uniformly bonded in variation sequence of finite topological measures
has a subsequence which is weakly fundamental.
\end{theorem}

\begin{proof}
Suppose $(\mu_n) \in \mathbf{DTM}(X)$ and $ \| \mu_n \| \le M$ for each $n$. Let $g \in C_0^+(X)$,  so $0 \le g \le b$ for some $b$. 
Each of the functions $R_{2, \mu_n, g} (t) $ is  monotone and  bounded above by $M$ on $[0,b]$. 
By the Helly-Bray theorem (see \cite[Theorem 1.4.6]{Bogachev:WkConv}), 
there is pointwise convergent subsequence $R_{2, \mu_{n_i}, g} $. Then the sequence 
of integrals $ \int_X g \, d \mu_{n_i} = \int_0^b R_{2, \mu_{n_i}, g}(t) dt $ 
converges, hence, is fundamental.  

If $G$ is a countable dense set in  $C_0^+(X)$, we pick a first subsequence of $ (n_i)$ such that 
$ ( \int_X g_1 \, d \mu_{n_i}) $ is fundamental for the first function $g_1 \in G$,
then we choose a further subsequence $(n_{i_j})$ for which 
$ ( \int_X g_2 \, d \mu_{n_{i_j}}) $ is fundamental for the function $g_2 \in G$, and so on. 
By diagonal process we obtain a subsequence of $(\mu_n)$ for which 
the sequence of integrals is fundamental for each $ g \in G$. For notational simplicity, let us assume that $(\mu_n)$ is such a subsequence, i.e. 
$ ( \int_X g \, d \mu_n) $ is fundamental for each function $g \in G$.  

For arbitrary $ f \in  C_0^+(X)$ and $ \epsilon >0$ choose $ g \in G$ such that $ \| f - g \| \le \epsilon$ 
and $n_0$ such that $ | \int_X g \, d\mu_n - \int_X g \, d \mu_i | < \epsilon$ for $n , i \ge n_0$.
Then using \cite[Corollary 53]{Butler:QLFLC} we have:
\begin{align*} 
| &\int_X f \, d\mu_n - \int_X f \, d \mu_i |  \\
&\le | \int_X f \, d\mu_n - \int_X g \, d \mu_n |  +  | \int_X g \, d\mu_n - \int_X g \, d \mu_i |  +  
| \int_X g \, d\mu_i - \int_X f \, d \mu_i |  \\
&\le \| f - g \| \| \mu_n \| + \epsilon +  \| f - g \| \| \mu_i \|   \le 2 \epsilon M + \epsilon,
\end{align*}  
and the sequence of integrals $ ( \int_X f \, d \mu_n) $ is fundamental. Thus, $ ( \mu_n)$ is weakly fundamental.
\end{proof}  

\begin{remark} 
If $X$ is a locally compact Hausdorff space which is second countable or satisfies any of the other equivalent conditions 
of \cite[Theorem 5.3, p.29]{Kechris}, then 
$\hat{X}$, the Aleksandrov one-point compactification of $X$, is a compact metrizable (hence, a second countable) space. 
Then $C(\hat{X})$ is separable, and $C_0(X)$ is also separable as as a subspace of a separable metric space.
\end{remark} 

For topological measures we have the following version of Prokhorov's well-known  theorem.

\begin{theorem} \label{Proh2}
Suppose $X$ is a locally compact space such that $C_0^+(X)$ is separable. 
Suppose $\mathcal{M}$ is a family of finite topological measures on $X$.
The the following are equivalent:
\begin{enumerate}
\item \label{ProhUsl1}
If every sequence from $\mathcal{M}$ contains a weakly convergent subsequence then $\mathcal{M}$ is uniformly tight and uniformly bounded in variation.
\item \label{ProhUsl2}
If $\mathcal{M}$ is uniformly bounded in variation then every sequence from $\mathcal{M}$ contains a weakly convergent subsequence. 
\end{enumerate}
\end{theorem}

\begin{proof}
(\ref{ProhUsl1}) follows from Theorem \ref{wkfamBd} and Theorem \ref{wkfamTgt}. 
(\ref{ProhUsl2}) follows from Theorem \ref{EbSmul} and Lemma \ref{wkfuncon}.
\end{proof}

\section{Prokhorov and Kantorovich-Rubenstein metrics}

It is clear that $d_o(\mu, \nu) = \sup \{| \int_X f \, d\mu -  \int_X f \, d\nu|: f \in C_0^+(X) \} $ is a metric on  $ \mathbf{DTM}(X)$, and the topology
induced by this metric is the weak topology. 
 
For the rest of this section let $(X, d)$ be a locally compact metric space. We shall consider two other metrics on $ \mathbf{DTM}(X)$. 

Let $A^t = \{ x \in X: d(x, A) < t \}$ for $ A \in \mathscr{O}(X) \cup \mathscr{C}(X), A \ne \emptyset$, and $\emptyset^t = \emptyset$ 
for all $ t >0$.
Each $A^t$ is an open set.
Consider the Prokhorov metric $d_{\text{P}}$ on  $ \mathbf{DTM}(X)$:
\begin{align*}
 d_{\text{P}}(\mu, \nu) &= \inf\{ t >0: \, \mu(A) \le \nu(A^t) + t, \ \nu(A) \le \mu(A^t) + t,  \\
&\forall A \in \mathscr{O}(X) \cup \mathscr{K}(X) \}.
\end{align*}
Taking $ t = \| \mu\| +   \| \nu\| $ we see that  $\inf$ is well defined.

Note that if $\mu$ and $ \nu$ are Borel measures and $A$ is a Borel set, then we obtain the usual definition of Prokhorov's metric 
(sometimes also called L\'{e}vy-Prokhorov metric).   

\begin{lemma} \label{dPmetric}
$d_{\text{P}}$  is a metric on $ \mathbf{DTM}(X)$.
\end{lemma}

\begin{proof}
It is clear that $d_{\text{P}} \ge 0$ and $ d_{\text{P}}(\mu, \nu) = d_{\text{P}}( \nu, \mu)$. 
For any $ A \in \mathscr{O}(X) \cup \mathscr{C}(X)$ we have $ \mu(A) \le \mu(A^t )+ t$ for all $ t >0$, so $d_{\text{P}}(\mu, \mu) = 0$.
Suppose $ d_{\text{P}}(\mu, \nu) = 0$.Then there is $ t_n \searrow 0$ such that 
$ \mu(K) \le \nu(K^{t_n}) + t_n$ and $ \nu(K) \le \mu(K^{t_n}) + t_n$
for all $ K \in \mathscr{K}(X)$.
For $K \in \mathscr{K}(X)$ and $ \epsilon >0$ choose $U \in \mathscr{O}(X)$ such that $ K \subseteq U$ and $\nu(U) < \nu(K) + \epsilon$. 
There exists $ r>0$ such that $K^r\subseteq U$. 
Then for $ t_n < r$
$$ \mu(K) \le \nu(K^{t_n}) + t_n \le \nu(U) + t_n  \le \nu(K)  + \epsilon +  t_n.$$
It follows that $ \mu(K) \le \nu(K)$, and, similarly, $ \nu(K) \le \mu(K)$. Then $ \mu = \nu$ on $ \mathscr{K}(X)$, so $ \mu = \nu$.

Now we shall show the triangle inequality. Suppose that for all $A \in \mathscr{O}(X) \cup \mathscr{K}(X)$
$$ \mu(A) \le \lambda(A^t)  + t, \quad \lambda(A) \le \mu(A^t)  + t,$$
$$ \lambda(A) \le \nu(A^r)  + r, \quad  \nu(A) \le \lambda(A^r)  + r. $$
Since $(A^t)^r \subseteq A_{t + r}$ and $ (A^r)^t \subseteq A_{t + r}$, we have:
$$ \mu(A) \le \lambda(A^t)  + t \le \nu(A^t)^r + t + r \le \nu(A_{t + r}) + t + r,$$  
and, similarly, $ \nu(A)  \le \mu(A_{t + r}) + t + r$. Thus, $ d_{\text{P}}(\mu, \nu) \le t + r$. It follows  that
$ d_{\text{P}}(\mu, \nu) \le d_{\text{P}}( \mu, \lambda) + d_{\text{P}}(\lambda, \nu)$.
\end{proof}

\begin{theorem} \label{dPwkconv}
Let  $(X, d)$ be a locally compact metric space. 
Suppose $ d_{\text{P}}( \mu_{\alpha}, \mu) \rightarrow 0$ for a net  $(\mu_{\alpha})$; $ \mu_{\alpha}, \mu \in \mathbf{DTM}(X)$. 
Then $ \mu_{\alpha} \Longrightarrow \mu$.
\end{theorem}

\begin{proof}
Suppose $ d_{\text{P}}( \mu_{\alpha}, \mu) \rightarrow 0$. 

Let $K \in \mathscr{K}(X)$ and $\epsilon >0$. Choose $ U \in \mathscr{O}(X)$ such that $ K \subseteq U$ and $ \mu(U) < \mu(K) + \epsilon $. 
There exists $ r>0$ such that  $K^t \subseteq U$ for all $t \le r$. For $ \delta = min \{r, \epsilon \}$ let $ \alpha_0$ be such that 
$d_{\text{KR}}( \mu_{\alpha}, \mu) < \delta$ for each $ \alpha \ge \alpha_0$. 
Then for each  $ \alpha \ge \alpha_0$ there exists $t_{\alpha} < \delta$ such that 
$ \mu_{\alpha}(K) \le \mu(K^{t_{\alpha}}) + t_{\alpha} \le \mu(U) + \epsilon \le \mu(K) + 2 \epsilon$.
Then 
$$ \limsup \mu_{\alpha}(K) \le \mu(K) + 2 \epsilon.$$
It follows that $ \limsup \mu_n(K) \le \mu(K)$.

Now let $U \in \mathscr{O}(X)$ and $\epsilon >0$. Choose $ K \in \mathscr{K}(X)$  such that $ K \subseteq U$ and $ \mu(K) > \mu(U) - \epsilon $. 
Let $r, \delta$ and $ \alpha_0$ be as above.
Then for each  $ \alpha \ge \alpha_0$ there exists $t_{\alpha} < \delta$ such that 
$ \mu(K)  \le \mu_{\alpha}(K^{t_{\alpha}}) + t_{\alpha} \le \mu_{\alpha} (U) + \epsilon$. Then 
$$ \liminf \mu_{\alpha}(U) \ge  \mu(K) - \epsilon \ge \mu(U) - 2 \epsilon.$$
It follows that $ \liminf(U) \ge \mu(U)$. 

By Theorem \ref{AleksandrovLC}   $ \mu_{\alpha}  \Longrightarrow \mu$.
\end{proof}

Let family $\mathcal{M}  \subseteq \mathbf{TM}(X)$ be uniformly bounded in variation.
We consider the Kantorovich-Rubinstein metric $d_{\text{KR}}$ on  $\mathcal{M}$.
\begin{align} \label{KRmetric}
d_{\text{KR}}(\mu, \nu) &= \sup \{| \int_X f \, d\mu -  \int_X f \, d\nu|: f \in Lip_1(X,d) \cap C_c(X), \, \| f \| \le 1 \}  
\end{align}
where $Lip_1(X) = \{f:X \Longrightarrow \mathbb{R}: | f(x) - f(y) | \le d(x,y) \ \forall x,y \in X\} $. 

\begin{remark}
Our definition is related to the definition of the Kantorovich-Rubinstein metric
for Borel measures, which is obtained from the Kantorovich-Rubinstein norm
$$ \| \mu \|_{KR} = \sup \{ \int_X f \, d\mu: f \in Lip_1(X,d), \, \| f \| \le 1 \}.$$ 
This metric is sometimes is also called 
the Wasserstein metric $W(\mu, \nu)$, although there is no author with this name. See \cite[pp. 453-454, Comments to Ch.8]{Bogachev} 
for a good note on the history and use of this metric.

Our use of $f \in Lip_1(X,d) \cap C_c(X)$ in (\ref{KRmetric}) is dictated, on one hand, by relation to Kantorovich-Rubinstein metric
for Borel measures and, on the other hand, by the role of $C_c(X)$ in the theory of (p-conic) quasi-linear functionals.
Note that by \cite[Theorem 2]{Andreou} Lipschitz functions with compact support are dense in $C_0(X)$.   
\end{remark}   
  
\begin{lemma}
 $d_{\text{KR}}$ is a metric on a uniformly bounded in variation family $\mathcal{M}$.
\end{lemma}

\begin{proof}
We shall show that $d_{\text{KR}}(\mu, \nu) = 0$ implies $\mu = \nu$; the remaining properties are obvious. 
Let $M$ be such that $ \| \mu \| \le M$ for each $\mu \in \mathcal{M}$.  
Take $f \in C_0(X)$. Given $ \epsilon >0$, choose a Lipschitz function $g$ with compact support so that $ \| f - g \| < \epsilon$. 
Since  $d_{\text{KR}}(\mu, \nu) = 0$, we see that $ | \int_X g \, d\mu - \int_X g \, d \nu |  = 0$. Using also Remark \ref{LipQLF}  we have:
\begin{align*} 
| &\int_X f \, d\mu - \int_X f \, d \nu |  \\
&\le | \int_X f \, d\mu - \int_X g \, d \mu |  +  | \int_X g \, d\mu - \int_X g \, d \nu |  +  
| \int_X g \, d \nu - \int_X f \, d \nu |  \\
&\le \| f - g \|  \mu(X)  + \| f - g \|  \nu(X)   \le 2 \epsilon M.
\end{align*}  
Thus, $\int_X f \, d\mu = \int_X f \, d \nu$ for every $f  \in C_0(X)$. By Remark \ref{RemBRT}  $ \mu = \nu$.
\end{proof}

\begin{theorem} \label{dKRwkconv}
Let  $X$ be a locally compact metric space. 
In either of the following situations:
\begin{enumerate}
\item
a family $\mathcal{M}  \subseteq \mathbf{TM}(X)$ is uniformly bounded in variation; 
\item
given $M >0$, a family $\mathcal{M}  \subseteq \mathbf{DTM}(X)$ is the family of deficient topological measures corresponding 
to functionals  $\mathcal{R}$ on $C_c^+(X)$ with $ \| \mathcal{R} \| \le M$;
\end{enumerate}
if a net $(\mu_{\alpha}) \in \mathcal{M} $, $\mu \subseteq  \mathcal{M} $, and $ d_{\text{KR}} ( \mu_{\alpha}, \mu) \rightarrow 0$, 
then $ \mu_{\alpha} \Longrightarrow \mu$. 
\end{theorem}

\begin{proof}
\begin{enumerate}
\item
Let  $f \in C_0(X)$. Given $ \epsilon >0$, choose a Lipschitz function with compact support $g$ so that $ \| f - g \| < \epsilon$. 
Since$ | \int_X g \, d \mu_{\alpha} - \int_X g, d \mu | \le d_{\text{KR}} ( \mu_{\alpha}, \mu) \,  \| g \|_{Lip} \,  \| g \|$, 
say,  $ | \int_X g \, d \mu_{\alpha} - \int_X g, d \mu | \le  \epsilon $ for all $ \alpha \ge \alpha_0$.  
Then for all $ \alpha \ge \alpha_0$ using Remark \ref{LipQLF} we have: 
\begin{align*} 
| &\int_X f \, d\mu_{\alpha} - \int_X f \, d \mu |  \\
&\le | \int_X f \, d\mu_{\alpha}- \int_X g \, d\mu_{\alpha}|  +  | \int_X g \, d\mu_{\alpha}- \int_X g \, d \mu |  +  
| \int_X g \, d\mu - \int_X f \, d \mu |  \\
&\le \| f - g \| \mu_{\alpha}(X) + \epsilon +  \| f - g \| \mu(X)  \le 2 \epsilon M + \epsilon,
\end{align*}  
so  $\int_X f \, d\mu_{\alpha} \longrightarrow  \int_X f, d\mu$. It follows that  $ \mu_{\alpha} \Longrightarrow \mu$.
\item
If a deficient topological measure corresponds to  $\mathcal{R}$  then  $ \| \mu \| \le M$. 
Thus, the family $\mathcal{M} $ is uniformly bounded in variation, and we may use the same argument as in previous part.
\end{enumerate}
\end{proof} 

\begin{theorem} 
Let $X$ be a compact metric space. Given $M >0$, let $ \mathcal{M} = \{ \mu \in \mathbf{DTM}(X):  \| \mu \| \le M \}$.
Then the topology on $ \mathcal{M}$ induced  by the metric $ d_{\text{KR}}$ is the weak topology.
\end{theorem}

\begin{proof}
By Theorem \ref{dKRwkconv} if a net $(\mu_{\alpha})$ converges to $\mu$ in the metric $d_{\text{KR}}$ then it also converges to $\mu$ weakly. 
For $ \mathcal{M} = \{ \mu \in \mathbf{TM}(X):  \| \mu \| \le 1 \}$ and a slightly different metric 
the result was first shown in \cite[Proposition 1.10]{DickZap}, and our proof of 
Theorem \ref{dKRwkconv} follows the argument in that paper.
Because of Remark \ref{LipQLF} and the fact that the family of functions in (\ref{KRmetric}) is compact by the Arzela-Ascoli theorem, 
one can basically repeat an argument from \cite[Proposition 1.10]{DickZap} to show that the weak convergence of
 $(\mu_{\alpha})$ to $\mu$ implies convergence in the metric $d_{\text{KR}}$. 
 \end{proof}  
               
\section{Density theorems}  

\begin{definition} \label{properSDTM}
A deficient topological measure $\nu$  is called proper if from $m \le \nu $, 
where $m$ is a Radon measure it follows that $m = 0$.
\end{definition}

\begin{remark} \label{properDtm}
From \cite[Theorem 4.3]{Butler:Decomp} it follows that a finite deficient topological measure can be written as a sum of a finite Radon measure and 
a proper finite deficient topological measure. The sum of two proper deficient topological measures is proper (see \cite[Theorem 4.5]{Butler:Decomp}). 

A finite Radon measure on a compact space is a regular Borel measure, so our definition (which is given in \cite{Butler:Decomp})
of a proper deficient topological measure coincides with definitions in papers prior to \cite{Butler:Decomp}.
\end{remark}

In what follows, $p\mathbf{DTM}(X)$ and  $p\mathbf{TM}(X)$  denote, respectively, the family of proper finite deficient topological measures and the family 
of finite topological measures. 

Let $X$ be a locally compact non-compact space.
A set $A$ is called solid if $A$ is  connected, and $X \setminus A$ has only unbounded connected components.
When $X$ is compact, a set is called solid if it and its complement are both connected.
For a compact space $X$ we define a certain topological characteristic, genus. 
See \cite{Aarnes:ConstructionPaper} for more information about genus $g$ of the space. 
A compact space has genus 0 iff any finite union of disjoint closed solid sets has a connected complement.
Intuitively, $X$ does not have holes or loops.
In the case where $X$ is locally path connected, $g=0$ if the fundamental group $\pi_1(X)$ is finite (in particular, if $X$ is 
simply connected). Knudsen \cite{Knudsen} was able to show that if 
$H^1(X) = 0 $ then $g(X) = 0$, and in the case of CW-complexes the converse also holds.

\begin{remark} \label{notTM}
From Theorem \ref{DTMtoTM} it is easy to see that
if $ \mu, \nu$ are deficient topological measures, and $\nu$ is not a topological measure, then $ \mu + \nu$ is a deficient topological measure which is not a 
topological measure. 
\end{remark}

\begin{theorem} \label{messyTh}
\begin{enumerate}
\item \label{paforDTM}
(Proper simple deficient topological measures that are not topological measures are dense in the set of all point-masses) $\Longrightarrow$
($p\mathbf{DTM}(X) \setminus \mathbf{TM}(X)$ is dense in $\mathbf{M}(X)$)  $\Longleftrightarrow$ ($p\mathbf{DTM}(X) \setminus \mathbf{TM}(X)$ 
is dense in $\mathbf{DTM}(X) \setminus \mathbf{TM}(X)$) 
$\Longrightarrow$ ($p\mathbf{DTM}(X) $ is dense in $\mathbf{DTM}(X)$)  $\Longleftrightarrow$ ($p\mathbf{DTM}(X) $ is dense in $\mathbf{M}(X)$). 
\item \label{paforTM}
(Proper simple $\mathbf{TM}(X)$ are dense in the set of all point-masses)  $\Longrightarrow$ ($p\mathbf{TM}(X)$ is dense in $\mathbf{M}(X)$) $\Longleftrightarrow$ 
($p\mathbf{TM}(X)$ is dense in $\mathbf{TM}(X)$)
$\Longrightarrow$   ($p\mathbf{DTM}(X) $ is dense in $\mathbf{DTM}(X)$).
\end{enumerate}
\end{theorem}

\begin{proof}
We shall prove the first part; the proof of the second part is similar, but simpler.
\begin{enumerate}[label=(\Alph{enumi}),ref=(\Alph{enumi})] 
\item \label{imA}
We shall show the first implication.
Any measure is approximated by convex combinations of point-masses, so by assumption, it is approximated by convex combinations of
proper simple deficient topological measures that are not topological measures. By Remark \ref{properDtm} and Remark \ref{notTM}
the latter combinations are in $p\mathbf{DTM}(X) \setminus \mathbf{TM}(X)$. 
\item \label{imB} 
($p\mathbf{DTM}(X) \setminus \mathbf{TM}(X)$ is dense in $\mathbf{M}(X)$)  $\Longrightarrow$ ($p\mathbf{DTM}(X) \setminus \mathbf{TM}(X)$ 
is dense in $\mathbf{DTM}(X) \setminus \mathbf{TM}(X)$):
Suppose $ \mu \in \mathbf{DTM}(X) \setminus \mathbf{TM}(X)$. By Remark \ref{properDtm}  write $ \mu = m + \mu'$, 
where $ \mu'$ is a proper deficient topological measure,
and $ m$ is a measure from $\mathbf{M}(X)$. 
By assumption, $m$ is approximated by $\nu  \in p\mathbf{DTM}(X) \setminus \mathbf{TM}(X)$. Then $\mu$ is approximated by $\nu + \mu'$,  where  
by Remark \ref{properDtm} and 
Remark \ref{notTM}   $\nu + \mu' $ is in $p\mathbf{DTM}(X) \setminus \mathbf{TM}(X)$.  
\item \label{imC}
($p\mathbf{DTM}(X) \setminus \mathbf{TM}(X)$ is dense in $\mathbf{DTM}(X) \setminus \mathbf{TM}(X)$) 
$\Longrightarrow$ ($p\mathbf{DTM}(X) \setminus \mathbf{TM}(X)$ 
is dense in $\mathbf{M}(X)$): 
Suppose to the contrary that there exists a measure $m \in \mathbf{M}(X) $ and its neighborhood $N$ which contains 
no elements of $p\mathbf{DTM}(X) \setminus p\mathbf{TM}(X)$.
Take $\lambda \in \mathbf{DTM}(X) \setminus \mathbf{TM}(X)$. Then for any deficient topological measure $ \nu \in N$  
we see that $ \lambda + \nu$ is a deficient topological measure 
that is not a topological measure and is not proper. Thus, a neighborhood  $\lambda + N \subseteq\mathbf{DTM}(X) \setminus \mathbf{TM}(X)$ contains no elements of 
$p\mathbf{DTM}(X) \setminus \mathbf{TM}(X)$, which contradicts the assumption. 
\item \label{imD}
($p\mathbf{DTM}(X) \setminus \mathbf{TM}(X)$ is dense in $\mathbf{DTM}(X) \setminus \mathbf{TM}(X)$) $\Longrightarrow$ ($p\mathbf{DTM}(X) $
is dense in $\mathbf{DTM}(X)$):
Let $\nu \in \mathbf{DTM}(X)$. If $ \nu \in \mathbf{DTM}(X) \setminus p\mathbf{TM}(X)$ then the statement follows from the assumption, and 
if  $ \nu \in \mathbf{DTM}(X) \cap p\mathbf{TM}(X)$ then the statement is obvious. 
\item \label{imE}
($p\mathbf{DTM}(X) $ is dense in $\mathbf{DTM}(X)$)  $\Longrightarrow$ ($p\mathbf{DTM}(X) $ is dense in $\mathbf{M}(X)$): 
obvious.
\item \label{imF}
($p\mathbf{DTM}(X) $ is dense in $\mathbf{M}(X)$) $\Longrightarrow$ ($p\mathbf{DTM}(X) $ is dense in $\mathbf{DTM}(X)$): follows from 
Remark \ref{properDtm} and Remark \ref{notTM} in a manner similar to the one in part \ref{imB}. 
\end{enumerate}
\end{proof}

\begin{theorem} \label{Kcden}
Suppose any open set in a locally compact space $X$ contains a compact connected subset that is not a singleton. 
Then $p\mathbf{DTM}(X)$ is dense in $\mathbf{DTM}(X)$.
\end{theorem}

\begin{proof}
If we shall show that proper simple $\mathbf{DTM}(X) \setminus \mathbf{TM}(X)$ are dense in the set of point-masses, 
then the statement will follow from Theorem \ref{messyTh}.
Let $\delta_a$ be a point-mass at $a$. Let $ \{ V \in \mathscr{O}(X): a \in V\}$ be ordered by reverse inclusion. For each $V$, 
let $K_V \subseteq V$ be the non-singleton connected compact set.
Consider $ \lambda^V$  defined on $\mathscr{O}(X) \cup \mathscr{C}(X)$ as follows: $ \lambda^V (A) =1$ if $ K_V \subseteq A$ and  $ \lambda^V (A) =0$ otherwise. 
By \cite[Example 46]{Butler:DTMLC} $ \lambda^V$ is simple and $ \lambda^V \in \mathbf{DTM}(X) \setminus \mathbf{TM}(X)$.
If $ U \in \mathscr{O}(X) $ and $ \delta_a(U) = 1$, then $ a \in U$ and for all 
$ V \subseteq  U, V \in \mathscr{O}(X)$ we have $ K_V \subseteq U$, so $ \lambda^V (U) = 1$. Then $ \liminf  \lambda^V (U)  = 1  = \delta_a(U)$.
If $C \in \mathscr{K}(X)$ and $ \delta_a(C) = 0$, then $a \notin C$ and we may find $U \in \mathscr{O}(X) $ such that $a \in U, U \cap C = \emptyset$. 
Then for each $ V \subseteq U, V \in \mathscr{O}(X)$ we have $K_V \cap C = \emptyset$ and $\lambda^V(C) = 0$. Then $ \limsup \lambda^V(C)  = 0 = \delta_a(C)$. 
By Theorem  \ref{AleksandrovLC} the net $(\lambda^V)$ converges weakly to $ \delta_a$. 
\end{proof}

\begin{remark} \label{KcdenRe}
Among spaces that satisfy the condition of the previous theorem are: non-singleton locally compact spaces that are locally connected or
weakly locally connected; manifolds; CW complexes.
\end{remark}

\begin{theorem} \label{singDen}
Suppose $X$ is a non-singleton connected, locally connected, locally compact space with no cut points and 
such that the Aleksandrov one-point compactification of $X$ has genus $0$.
Then $p\mathbf{TM}(X)$ is dense in  $\mathbf{TM}(X)$, and  $p\mathbf{DTM}(X)$ is dense in  $\mathbf{DTM}(X)$.
\end{theorem}

\begin{proof}
We shall give the proof for the case when $X$ is not compact. (When $X$ is compact the proof is similar but simpler; 
also, one may use  \cite[Theorem 4.9]{Svistula:Integrals}.)
We shall show that proper simple topological measures are dense in the set of simple measures, and the statements will follow from part (\ref{paforTM}) of
Theorem \ref{messyTh}.

Let $\delta_a$ be a point-mass. It is enough to show that a neighborhood of the form $ W(\delta_a, U, C, \epsilon) $ as in Theorem \ref{wkBase} contains a
simple proper topological measure. 

Suppose first $a \in U \in \mathscr{O}(X), a \notin C$. We may assume that $ U \cap C = \emptyset$. 
Since $a \in U \in \mathscr{O}(X)$, by Lemma \ref{easyLeLC}
there is a bounded open connected set $V$ and a compact connected set $D$ such that $ a \in V \subseteq D \subseteq U$.  
Since $X$ is connected and non-singleton, $ a \subsetneq V $, and we may 
choose 3 different points in $D$. Let $\lambda$ be a simple topological measure on $X$ given by \cite[Example 46]{Butler:TMLCconstr}, so 
$\lambda(A) = 1$ if a bounded solid set $A$ contains two or three of the chosen points, and  
$\lambda(A) = 0$ if a bounded solid set $A$ contains no more than one of the chosen points.
Since the solid hull of $D$ (a compact solid set) contains all three points, and each bounded component of $X \setminus D$ (a bounded open solid set) contains none of 
the three points, by \cite[Definition 41]{Butler:TMLCconstr} we compute $\lambda(D) = 1$. Then $ \lambda(U) = 1$. Since $C$ is disjoint from $U$, 
and $ \lambda(X) = 1$, by superadditivity we have $ \lambda(C) = 0$. Thus, $ \lambda \in W(\delta_a, U, C, \epsilon) $.

We shall show that $ \lambda$ is proper. Let $x \in X$. Since $X \setminus \{x\}$ is connected, by Lemma \ref{easyLeLC} there is 
a compact connected set $B \subseteq X \setminus \{x\} $ such that $B$ contains at least two of the three chosen points. Argument as above shows that $\lambda(B) = 1$. 
Then $ \lambda(\{x\})  \le \lambda(X \setminus B) = \lambda(X) - \lambda(B) =0$. Thus, $\lambda(\{x\}) =0$ for any $x \in X$, and by \cite[Lemma 4.12]{Butler:Decomp} $\lambda$ is proper.

The remaining three cases are easy. For example, if $ a \in U, a \in C$ then $\lambda$ as above will do.
\end{proof}

\begin{lemma} \label{inftsumD}
Suppose $X$ is locally compact,  $\sum_{i=1}^\infty \mu_i(X) < \infty$ where each $\mu_i$ is a deficient topological measure. 
Then $\mu = \sum_{i=1}^\infty \mu_i$ is a finite deficient topological measure. If each $\mu_i$ is a topological measure, then $\mu$ is a finite topological measure.
\end{lemma}
 
\begin{proof}
Let $\mu = \sum_{i=1}^\infty \mu_i$  on $ \mathscr{O}(X) \cup \mathscr{C}(X)$. It is easy to see that $\mu$ is finitely additive on compact sets. For $ \epsilon >0$ let $j$ be such that
 $\sum_{i=j+1}^\infty \mu_i(X) < \epsilon$, and let $\lambda = \sum_{i=1}^j \mu_i$. Then $\lambda $ is a finite deficient topological measure.
 For $U \in \mathscr{O}(X)$ there exists $K \in \mathscr{K}(X)$ such that $\lambda(U)< \lambda(K) + \epsilon$. Then $ \mu(U) < \lambda(U) + \epsilon< \lambda(K) + 2\epsilon <\mu(K) + 2\epsilon$, and the inner regularity
 of $\mu$ follows. Similarly, $\mu$ is outer regular. Thus, $\mu$ is a deficient topological measure; clearly, $\mu$ is finite.
 If each  $\mu_i$ is a topological measure, it is easy to check
 additivity of $ \mu$  on $\mathscr{O}(X) \cup \mathscr{K}(X)$, so condition \ref{TM1} of Definition \ref{TMLC} holds, and $\mu$ is a topological measure.
 \end{proof}

\begin{lemma} \label{infprsum}
Suppose $X$ is locally compact,  $\sum_{i=1}^\infty \mu_i(X) < \infty$ where each $\mu_i$ is a proper deficient topological measure 
(respectively, a proper topological measure).
Then $\mu = \sum_{i=1}^\infty \mu_i$ is a finite  
proper deficient topological measure (respectively, a  finite proper  topological measure).
\end{lemma}
 
\begin{proof}
By Lemma \ref{inftsumD} $\mu$ is a finite deficient topological measure (respectively, a finite topological measure). We need to show that $\mu$ is proper.
By Remark \ref{properDtm} write $ \mu = m + \mu'$,
where $m$ is a finite Radon measure and $ \mu'$ is a proper deficient topological measure. We shall show that $ m=0$.

Let $K  \in \mathscr{K}(X)$.
For $ \epsilon >0$ let $N$ be such that
$\sum_{i=N+1}^\infty \mu_i(X) < \epsilon$, and let $\mu^N = \sum_{i=1}^N \mu_i$.  

By Remark \ref{properDtm} 
$\mu^N$ is a proper deficient topological measure.  By  \cite[Theorem 4.4]{Butler:Decomp}    
 there are compact sets $K_1, \ldots, K_n$ such that 
$K = \cup K_j$ and $ \sum_{j=1}^n \mu^N(K_j) < \epsilon$. 
Let  $E_1, \ldots, E_n $ be disjoint Borel sets such that $E_j \subseteq K_j$ and $ \bigsqcup_{j=1}^n E_i = \bigcup_{j=1}^n K_j$.
Since $m$ is finite,  outer regularity of $m$ is equivalent to inner closed regularity of $m$. Find disjoint sets $C_j, C_j  \subseteq E_j \subseteq K_j,  j=1, \ldots, n$ 
such that $C_j$ are closed (hence, compact) and 
$m(C_j) > m(E_j) - \frac{\epsilon}{n}$. Then
\begin{align*}
m&(K) = \sum_{j=1}^n m(E_j) \le \epsilon + \sum_{j=1}^n m(C_i) \le \epsilon +  \mu(C_1 \sqcup \ldots \sqcup C_n) \\
&\le \epsilon + \mu^N (C_1 \sqcup \ldots \sqcup C_n) + \epsilon =  2 \epsilon + \sum_{j=1}^n \mu^N(C_j ) \le 2 \epsilon + \sum_{j=1}^n \mu^N (K_i)  \le 3 \epsilon.
\end{align*}
It follows that $ m(K) = 0$ for any $ K \in \mathscr{K}(X)$.  Thus, $m=0$, and $\mu$ is proper.
\end{proof}

\begin{theorem}  \label{poXi}
Let $X$ be locally compact. Suppose $X = \bigcup_{i=1}^\infty X_i$, where each  $X_i $ is a compact subset of $X$. 
\begin{enumerate}
\item \label{poXi1}
If $p\mathbf{DTM}(X_i)$ is dense in $\mathbf{M}(X_i), \, i \in \mathbb{N}$ then  $p\mathbf{DTM} (X)$ is dense in $\mathbf{M} (X)$. 
\item \label{poXi2}
If $p\mathbf{TM} (X_i)$ is dense in $\mathbf{M}(X_i), \, i \in \mathbb{N}$ then  $p\mathbf{TM} (X)$ is dense in $\mathbf{M} (X)$. 
\end{enumerate}
\end{theorem}

\begin{proof}
Note that each $X_i$ is a locally compact space with respect to the subspace topology.
We  shall prove the first part.
Let $m \in \mathbf{M}(X)$.
We shall show that every neighborhood $W$ of $m$ as in Theorem \ref{wkBase} contains a proper deficient topological measure. 
To simplify notation, we consider $W(m, U, C, \epsilon)$ where $U \in \mathscr{O}(X), C \in \mathscr{K}(X), \epsilon >0$.
Take Borel subsets $Y_i$ of $X$ such that $ Y_i \subseteq X_i$ and $\bigsqcup_{i=1}^\infty Y_i = X$.  
Consider $m_i(B) = m(B \cap Y_i)$, where $B$ is a Borel set in $X_i$, $i \in \mathbb{N}$. 
It is easy to see that $m_i \in \mathbf{M}(X_i)$. 

Let $ \epsilon >0$. Let $U_i = U \cap X_i,  C_i = C \cap X_i,  \epsilon_ i  = \epsilon 2^{-i}$ for $ i \in \mathbb{N}$, 
so $U_i$ is open in $X_i$ and $ C_i$ is compact in $X_i$. 
By assumption, there is $\lambda_i \in p\mathbf{DTM}(X_i)$ such that 
$\lambda_i \in  W(m_i; U_i, C_i, X_i, \epsilon_i)$. 
Let $\nu_i$ be the extension of  $\lambda_i$ to $\mathscr{O}(X) \cup \mathscr{C}(X)$ given by 
$\nu_i(A) = \lambda_i(A \cap X_i)$ for $ A \in \mathscr{O}(X) \cup \mathscr{C}(X)$. 
It is easy to see that  $\nu_i $ is a deficient topological measure, and $ \nu_i(X) = \lambda_i(X_i) < \infty$.
Since $\lambda_i$ is proper, by  \cite[Theorem 4.4]{Butler:Decomp}
given $ \delta >0$ there are sets of the form $ V_j \cap X_i, V_j \in \mathscr{O}(X), j=1, \ldots, n$ 
such that  they cover $X_i$ and $\sum_{j=1}^n \lambda_i(V_j \cap X_i) < \delta$.
Then open sets $V_1, \ldots, V_n, X \setminus X_i$ cover $X$ and 
$\sum_{j=1}^n \nu_i (V_j) + \nu_i(X \setminus X_i) = \sum_{j=1}^n \lambda_i(V_j \cap X_i) < \delta$, and so $\nu_i$ is proper. 
Thus, $\nu_i \in  p\mathbf{DTM}(X)$ by  \cite[Theorem 4.4]{Butler:Decomp}.

 Since $ \sum_{i=1}^\infty \nu_i(X) = \sum_{i=1}^\infty \lambda_i(X_i) \le  \sum_{i=1}^\infty (m_i(X_i) + \epsilon_i) =  m(X) + \epsilon < \infty$, 
by Lemma \ref{infprsum}
$\nu = \sum_{i=1}^\infty \nu_i$ is a  finite proper deficient topological measure. 
We have:
$$ \nu(U)  = \sum_{i=1}^\infty \nu_i(U) =  \sum_{i=1}^\infty \lambda_i(U \cap X_i) >\sum_{i=1}^\infty (m_i(U \cap X_i)  - \epsilon_i) = m(U) - \epsilon, $$
$$ \nu(C) =  \sum_{i=1}^\infty \nu_i(C) =  \sum_{i=1}^\infty \lambda_i(C \cap X_i) < \sum_{i=1}^\infty (m_i(C \cap X_i)  + \epsilon_i) = m(C) + \epsilon. $$ 
Thus,  $\nu \in W(m, U, C, \epsilon)$.

The proof of the second part is the same, taking into account that $\lambda_i, \nu_i, \nu$ are proper topological measures.
\end{proof}

\begin{corollary} \label{poxiCor}
Let $X = \cup_{i=1}^\infty X_i$, where each $ X_i $ as in Theorem \ref{singDen}. 
Then $p\mathbf{TM}(X) $ is dense in $\mathbf{TM}(X)$, and $p\mathbf{DTM}(X) $ is dense in $\mathbf{DTM}(X)$.
\end{corollary}

\begin{proof}
By part \ref{paforTM} of Theorem \ref{messyTh}  it is enough to show that $p\mathbf{TM}(X) $ is dense in $\mathbf{M}(X)$. By Theorem \ref{singDen}, 
$p\mathbf{TM}(X_i) $ is dense in $\mathbf{M}(X_i)$ for each $i$, and we apply part \ref{poXi2} of Theorem \ref{poXi}.
\end{proof} 

\begin{remark}
In Corollary \ref{poxiCor} one may take, for example, a compact n-manifold, $n \ge 2$ as $X$, 
or $X$ that is covered by countably many sets homeomorphic to balls $B^n$ with varying $n \ge 2$.
\end{remark}

\begin{lemma}
$\mathbf{TM}(X) $ is a closed subset of $\mathbf{DTM}(X)$, and $\mathbf{M}(X) $ is a closed subset of $\mathbf{DTM}(X)$.
\end{lemma}

\begin{proof}
By Remark \ref{RemBRT}
$  \mu \in  \mathbf{TM}(X) $ iff $ \rho$ is a quasi-linear functional on $C_0(X)$, and  $\mu \in \mathbf{M}(X)$ iff $\rho$ is a linear functional on $C_0(X)$,
where $ \rho(f) =  \mathcal{R}_{\mu} (f^+) -  \mathcal{R}_{\mu} (f^-) $. Using basic open sets in Definition \ref{defwk} it is easy to check that 
$\mathbf{TM}(X) $ is a closed subset of $\mathbf{DTM}(X)$, and $\mathbf{M}(X) $ is a closed subset of $\mathbf{DTM}(X)$.
\end{proof} 

\begin{theorem} \label{densEq}
Suppose $X$ is locally compact.
The following are equivalent:
\begin{enumerate}
\item \label{densEq1}
$\mathbf{M}(X) $ is nowhere dense in $\mathbf{DTM}(X)$ (or in $\mathbf{TM}(X)$).
\item \label{densEq2}
There exists a finite deficient topological measure (respectively, a finite topological measure) that is not a measure.
\item \label{densEq3}
There exists a nonzero  finite proper deficient topological measure (respectively, nonzero finite proper  topological measure).
\end{enumerate}
\end{theorem}

\begin{proof}
(\ref{densEq1}) $\Longrightarrow$ (\ref{densEq2}) is obvious. (\ref{densEq2})  $\Longrightarrow$ (\ref{densEq3}): 
Let $\mu$ be a deficient topological measure that is not a measure.
By Remark \ref{properDtm} write $ \mu = m + \mu'$ where $m$ is a measure and $\mu' $ is a proper deficient topological measure. Then $ \mu' \neq 0$.
(\ref{densEq3})  $\Longrightarrow$ (\ref{densEq1}): Suppose $\nu \neq 0$ is a proper finite deficient topological measure. Let $ m \in \mathbf{M}(X) $. 
Consider  a set functions $\mu_n$ on $\mathscr{O}(X) \cup \mathscr{C}(X)$ given by 
$$\mu_n(A) = \frac1n \frac{1}{\nu(X)} \nu(A) + (1 - \frac1n) m(A).$$
Then each $\mu_n$ is a deficient topological measure that is not a measure, and  $ \mu_n \Longrightarrow m$  by Theorem \ref{AleksandrovLC}.
Thus, $\mathbf{DTM} (X) \setminus \mathbf{M}(X)$ is dense in $ \mathbf{M}(X)$, and since   $ \mathbf{M}(X)$  is a closed subset of   $\mathbf{DTM}(X)$, 
we see that $ \mathbf{M}(X)$ is nowhere dense in  $\mathbf{DTM}(X)$. 
The proof for topological measures is similar.
\end{proof}

\begin{corollary} \label{ExDensi} 
Suppose $X$ is locally compact.
If $X$ contains a non-singleton compact connected set, then $\mathbf{M}(X) $ is nowhere dense in $\mathbf{DTM}(X)$.
If $X$ contains an open (or closed) locally connected, connected, non-singleton subset whose Aleksandrov one-point compactification has genus $0$ 
then $\mathbf{M}(X) $ is nowhere dense in $\mathbf{TM}(X)$.
\end{corollary}

\begin{proof} 
Use part (\ref{densEq2}) of Theorem \ref{densEq}. 
For the first statement, as an example of a finite deficient topological measure that is not a topological measure (hence, not a measure)  
one may use \cite[Example 46]{Butler:DTMLC},
For the second statement, as an example of a finite topological measure that is not a measure one may take \cite[Example 61]{Butler:TMLCconstr}.
\end{proof}

The proof of the next Theorem and Corollary are similar to the proof of Theorem \ref{densEq} and Corollary \ref{ExDensi}.

\begin{theorem} \label{densEqT}
Suppose $X$ is locally compact.
The following are equivalent:
\begin{enumerate}
\item \label{densEq1a}
$\mathbf{TM}(X) $ is nowhere dense in $\mathbf{DTM}(X)$. 
\item \label{densEq2a}
There exists a finite deficient topological measure that is not a topological measure.
\item \label{densEq3a}
There exists a nonzero finite proper deficient topological measure that is not a topological measure.
\end{enumerate}
\end{theorem}

\begin{corollary} \label{CdensEqT}
If a locally compact space $X$ contains a non-singleton compact connected set, then $\mathbf{TM}(X) $ is nowhere dense in $\mathbf{DTM}(X)$.
\end{corollary}

\begin{remark}
When the space is compact, the equivalence of the first two conditions in Theorem \ref{AleksandrovLC} and of first three conditions in Theorem \ref{AleksandrovLCtm} 
was first given in \cite[Corollary 4.4, 4.5]{Svistula:Integrals}. 
When $X$ is compact Theorem \ref{wkBase} was proved in \cite{Svistula:Integrals}, but the method there does not work for a locally compact non-compact space, 
as the set $f^{-1}([0, \infty)) = X$ is not compact. 
Theorem \ref{basicTP} generalizes results from several papers, including \cite{Aarnes:QuasiStates70},  \cite{GrubbLaBerge:Spaces}, and  \cite{Svistula:Integrals}. 
Theorem \ref{metrization} is an adaptation of  \cite[Theorem 6.2]{Parthasarathy}.  
Our proof of Theorem \ref{wkfamTgt} is adapted from a nice proof in \cite[Theorem 2.3.4]{Bogachev:WkConv}.
In the last section we generalize results from \cite[Section 4]{Svistula:Integrals} and \cite{Density} from a compact space to a locally compact one. 
\end{remark}

{\bf{Acknowledgments}}:
The author would like to thank the Department of Mathematics at the University of California Santa Barbara for its supportive environment.

\section*{Conflict of interest}
 The author declares no conflict of interest.

\end{document}